\documentclass[reqno]{amsart}%
\usepackage{amsfonts, amsthm}
\usepackage{eurosym}
\usepackage{amsmath}
\usepackage{amssymb}
\usepackage{graphicx}
\usepackage{xcolor}
\usepackage{amsfonts}
\usepackage{enumitem}
\usepackage{subcaption}
\usepackage{float}
\usepackage{caption}%
\providecommand{\U}[1]{\protect\rule{.1in}{.1in}}
\providecommand{\U}[1]{\protect\rule{.1in}{.1in}}
\newtheorem{theorem}{Theorem}[section]
\theoremstyle{plain}
\newtheorem{corollary}[theorem]{Corollary}

\newtheorem{lemma}[theorem]{Lemma}

\newtheorem{remark}{Remark}[section]
\numberwithin{equation}{section}

\begin{document}
\title[$L^{p}$-convergence of Kantorovich type Max-Min Operators]{$L^{p}$-convergence of Kantorovich-type Max-Min Neural Network Operators}
\author[\.{I}SMA\.{I}L ASLAN, STEFANO DE MARCHI, WOLFGANG ERB]{\.{I}SMA\.{I}L ASLAN${^{\ast}}$, STEFANO DE MARCHI, WOLFGANG ERB}
\thanks{$^{\ast}$Corresponding author.}

\begin{abstract}
In this work, we study the Kantorovich variant of max-min neural network
operators, in which the operator kernel is defined in terms of sigmoidal
functions. Our main aim is to demonstrate the $L^{p}$-convergence of these
nonlinear operators for $1\leq p<\infty$, which makes it possible to obtain
approximation results for functions that are not necessarily continuous. In
addition, we will derive quantitative estimates for the rate of approximation
in the $L^{p}$-norm. We will provide some explicit examples, studying the
approximation of discontinuous functions with the max-min operator, and
varying additionally the underlying sigmoidal function of the kernel. Further,
we numerically compare the $L^{p}$-approximation error with the respective
error of the Kantorovich variants of other popular neural network operators.
As a final application, we show that the Kantorovich variant has advantages
compared to the sampling variant of the max-min operator and Kantorovich variant of the max-product operator when it comes to
approximate noisy functions as for instance biomedical ECG signals.

\end{abstract}
\maketitle

\section{Introduction}

Neural network operators have been widely studied in approximation theory
\cite{anas1, anas3, anas4, anas5,anas2, anas7,is,costez, cost1, cost3,
cost4,cost2, cost25, cost5,cheang, cao,cao2}. Although originally mainly
linear neural network operators have been considered
\cite{anas1,anas4,anas5,anas2,costez,cost1,cost3,cost4,coroianu4,kad},
nonlinear versions of these operators have recently begun to attract the
attention of several research groups
\cite{anas7,is,cost8,cost2,cost25,cost5,cost7,cost9}.

Bede et al., changing the algebraic structure of summation and multiplication,
constructed pseudo-linear operators, which are in fact nonlinear \cite{bede1}.
Further, pseudo-linear versions of Shepard operators were investigated, and it
was shown that these operators could outperform the classical ones, both, in
the rate of approximation, as well as in computational complexity
\cite{bede1,bede2}. These results were the starting point of several further
research on this topic and, as far as today, a lot of studies have been
conducted on pseudo-linear operators \cite{bede3,
bede4,bede5,coroianu1,coroianu2,coroianu3,duman2,duman,aslan,gokcer2,gokcer,holhos1,yuksel}%
.

Costarelli and Spigler in \cite{cost1} studied positive linear neural network
operators activated by sigmoidal functions which, for sufficiently large
$n\in\mathbb{N}$, are given in the form%

\begin{equation}
F_{n}\left(  f;x\right)  :=\dfrac{\sum\limits_{k=\left\lceil na\right\rceil
}^{\left\lfloor nb\right\rfloor }f\left(  \frac{k}{n}\right)  \phi_{\sigma
}\left(  nx-k\right)  }{\sum\limits_{d=\left\lceil na\right\rceil
}^{\left\lfloor nb\right\rfloor }\phi_{\sigma}\left(  nx-d\right)  }\text{
\ \ (}x\in\left[  a,b\right]  \text{)}, \label{1}%
\end{equation}
where $\phi_{\sigma}$ is a suitable linear combination of sigmoidal activation
functions $\sigma:\mathbb{R}\rightarrow\mathbb{R}$ and $f:\left[  a,b\right]
\rightarrow\mathbb{R}$ is a bounded function. In the definition above,
$\left\lfloor \cdot\right\rfloor $ and $\left\lceil \cdot\right\rceil $ denote
the floor and the ceiling function for a given number. Subsequently,
Costarelli and Vinti in \cite{cost2} constructed the (nonlinear) max-product
operators
\begin{equation}
F_{n}^{\left(  M\right)  } \left(  f;x\right)  :=\dfrac{\bigvee
\limits_{k=\left\lceil na\right\rceil }^{\left\lfloor nb\right\rfloor
}f\left(  \frac{k}{n}\right)  \phi_{\sigma}\left(  nx-k\right)  }%
{\bigvee\limits_{d=\left\lceil na\right\rceil }^{\left\lfloor nb\right\rfloor
}\phi_{\sigma}\left(  nx-d\right)  }\text{ \ \ (}x\in\left[  a,b\right]
\text{)}, \label{2}%
\end{equation}
with $\bigvee\nolimits_{k=1}^{n}a_{k}:=\max\nolimits_{k=1,\cdots,n}a_{k}$, and
derived a uniform approximation theorem.

Later, in \cite{cost4, cost25}, Costarelli and his colleagues investigated
also the Kantorovich form of the operators in (\ref{1}) and (\ref{2}).
Furthermore, in \cite{is}, the (nonlinear) max-min case of these operators was
examined and it was shown that compared to the positive linear neural network
operators in (\ref{1}) the max-min approximation provides better results in
some cases. It is noteworthy to mention that the max-min case, in which the
product is substituted by the minimum, has not been studied very profoundly so
far compared to other neural network operators. We also note that the max-min
operators are neither linear, nor homogeneous.

The aim of this article is to analyze the Kantorovich form of the max-min
neural network operators and to obtain a convergence theorem for $L^{p}%
$-spaces. This allows to consider such operators also for the approximation of
discontinuous functions. While the max-min variant of Kantorovich operators
has already been considered in a previous article \cite{aslan}, to the best of
our knowledge, this is the first study of the convergence of the max-min
operators in $L^{p}$ spaces. Furthermore, we also aim to obtain refined
estimates for the rates of approximation of these operators. The last part of
the article includes some applications to better illustrate the understanding
of the topic and to see in which scenarios the max-min Kantorovich form has advantages
compared to max-product Kantorovich variant and max-min sampling variant of the operator.

\section{Max-min neural network operators}

Let $f:\left[  a,b\right]  \rightarrow\left[  0,1\right]  $ be a given
function. Then, for $x\in\left[  a,b\right]  $, the max-min neural network
operator is defined as (see \cite{is}):%
\begin{equation}
F_{n}^{\left(  m\right)  }\left(  f;x\right)  :=\bigvee\limits_{k=\left\lceil
na\right\rceil }^{\left\lfloor nb\right\rfloor }f\left(  \frac{k}{n}\right)
\wedge\dfrac{\phi_{\sigma}\left(  nx-k\right)  }{\bigvee\limits_{d=\left\lceil
na\right\rceil }^{\left\lfloor nb\right\rfloor }\phi_{\sigma}\left(
nx-d\right)  }, \label{3}%
\end{equation}
where $a\wedge b$ denotes $\min\left\{  a,b\right\}  $. Before introducing the
Kantorovich type max-min neural network operators, we need some additional
definitions and assumptions.

For a given function $\sigma:\mathbb{R}\rightarrow\mathbb{R}$, we say that
$\sigma$ is a sigmoidal function if $\lim_{x\rightarrow-\infty}\sigma\left(
x\right)  =0$ and $\lim_{x\rightarrow\infty}\sigma\left(  x\right)  =1.$ For
the rest of the paper, we assume that $\sigma$ is a nondecreasing sigmoidal
function. We also assume that $\sigma\left(  3\right)  >\sigma\left(
1\right)  $, which is only a minor restriction that prevents some technical issues.

Beside the above definitions, the following conditions are needed.

\begin{enumerate}
\item[$\left(  \Sigma1\right)  $] $\sigma\left(  x\right)  -1/2$ is an odd
function on the real line,

\item[$\left(  \Sigma2\right)  $] $\sigma\in C^{2}\left(  \mathbb{R}\right)  $
is concave for all $x\in\mathbb{R}_{0}^{+},$

\item[$\left(  \Sigma3\right)  $] $\sigma\left(  x\right)  =O(\left\vert
x\right\vert ^{-(1+\alpha)})$ as $x\rightarrow-\infty$ for some $\alpha>0.$
\end{enumerate}

For the rest of this paper, we use the \textquotedblleft$\alpha$%
\textquotedblright\ symbol exclusively in connection to the condition
specified in $\left(  \Sigma3\right)  $. The kernel $\phi_{\sigma}$ (also
called \textquotedblleft centered bell shaped function\textquotedblright\ in
\cite{carda})\ in the definition of the neural network operator is given by%
\[
\phi_{\sigma}\left(  x\right)  :=\frac{1}{2}\left(  \sigma\left(  x+1\right)
-\sigma\left(  x-1\right)  \right)  \text{ \ for all }x\in\mathbb{R}\text{.}%
\]
We note that, by definition, the kernel $\phi_{\sigma}$ is not necessarily
compactly supported.

From the definitions and assumption above, it is possible to obtain the
following properties of $\phi_{\sigma}$ (see\cite{cost1}).

\begin{lemma}
\label{lemma1}

\begin{enumerate}
\item $\phi_{\sigma}\left(  x\right)  \geq0$ for all $x\in\mathbb{R}$ and
$\phi_{\sigma}\left(  2\right)  >0$,

\item $\lim_{x\rightarrow\pm\infty}\phi_{\sigma}\left(  x\right)  =0$,

\item $\phi_{\sigma}$ is nondecreasing if $x<0$ and nonincreasing if $x\geq0$
(therefore $\phi_{\sigma}\left(  0\right)  \geq\phi_{\sigma}\left(  x\right)
$ for all $x\in\mathbb{R)},$

\item $\phi_{\sigma}\left(  x\right)  =O(\left\vert x\right\vert ^{-\left(
1+\alpha\right)  })$ as $x\rightarrow\pm\infty$ where $\alpha$ refers to the
decay rate in $\left(  \Sigma3\right)  ,$ i.e., there exist $M,L>0$ such that
$\phi_{\sigma}\left(  x\right)  \leq M\left\vert x\right\vert ^{-\left(
1+\alpha\right)  }$ if $\left\vert x\right\vert >L.$

\item $\phi_{\sigma}\left(  x\right)  $ is an even function.

\item For all $x\in\mathbb{R},$ $\sum\limits_{k\in\mathbb{Z}}\phi_{\sigma
}\left(  x-k\right)  =1.$
\end{enumerate}
\end{lemma}

Note that from $(3)$ and $(6)$ of Lemma \ref{lemma1}, it is not hard to see
that $\phi_{\sigma}\in L^{1}\left(  \mathbb{R}\right)  $ (see also\ Remark 1
in \cite{cost25}). Furthermore, by the definition of $\sigma$ and
$\phi_{\sigma}$, we have $\phi_{\sigma}\left(  x\right)  \leq1/2$ for all
$x\in\mathbb{R}$.

\begin{remark}
\label{rem0} We note that, as in \cite{cost1, cost2}, the condition $\left(
\Sigma2\right)  $ is only used to prove Lemma \ref{lemma1} $(3)$ above. Thus,
the reader can infer that our theory is also valid for non-smooth sigmoidal
functions, which satisfy the property $(3)$ in Lemma \ref{lemma1} instead of
the condition $\left(  \Sigma2\right)  $.
\end{remark}

The following lemma is required for the well-definiteness of the Kantorovich
type operator given in (\ref{a}).

\begin{lemma}
(see \cite{cost2})\label{lemma2}

\begin{enumerate}
\item For a given $x\in\mathbb{R}$,
\[
\bigvee\limits_{k\in\mathbb{Z}}\phi_{\sigma}\left(  nx-k\right)  \geq
\phi_{\sigma}\left(  2\right)  >0
\]
holds for all $n\in\mathbb{N},$

\item Let the interval $\left[  a,b\right]  $ be given. Then for each
$x\in\left[  a,b\right]  $
\[
\bigvee\limits_{k=\left\lceil na\right\rceil }^{\left\lfloor nb\right\rfloor
-1}\phi_{\sigma}\left(  nx-k\right)  \geq\phi_{\sigma}\left(  2\right)  >0
\]
for sufficiently large $n\in\mathbb{N}$.
\end{enumerate}
\end{lemma}

If the index set has infinite elements, then \textquotedblleft$\vee
$\textquotedblright\ corresponds to the supremum.

\begin{lemma}
(see \cite{cost2})\label{lemma0} For every $\delta>0,$ we get%
\[
\bigvee\limits_{\overset{k\in\mathbb{Z}}{\left\vert x-k\right\vert >n\delta}%
}\phi_{\sigma}\left(  x-k\right)  =O\left(  n^{-\left(  1+\alpha\right)
}\right)  \text{ as }n\rightarrow\infty
\]
\newline uniformly in $x\in\mathbb{R}$.
\end{lemma}

We further state some well-known properties of the $\max$ and $\min$ operations.

\begin{lemma}
(see \cite{bede5})\label{lemma3} If $\bigvee\limits_{k\in\mathbb{Z}}%
a_{k}<\infty$ or $\bigvee\limits_{k\in\mathbb{Z}}b_{k}<\infty$, then there
holds
\[
\left\vert \bigvee\limits_{k\in\mathbb{Z}}a_{k}-\bigvee\limits_{k\in
\mathbb{Z}}b_{k}\right\vert \leq\bigvee\limits_{k\in\mathbb{Z}}\left\vert
a_{k}-b_{k}\right\vert .
\]

\end{lemma}

\begin{lemma}
(see \cite{bede1})\label{lemma4} For all $x,y,z\in\left[  0,1\right]  ,$ there
holds%
\[
\left\vert x\wedge y-x\wedge z\right\vert \leq x\wedge\left\vert
y-z\right\vert \text{,}%
\]
where $a\wedge b$ denotes the $\min\left\{  a,b\right\}  $.
\end{lemma}

\begin{lemma}
(see \cite{bede1})\label{lemma6} For any $a_{k}\geq0$ and $p\geq0,$ we have
\[
\biggl(\bigvee\limits_{k\in\mathbb{Z}}a_{k}\biggl)^{p}=\bigvee\limits_{k\in
\mathbb{Z}}a_{k}^{p}%
\]
and
\[
\biggl(\bigwedge\limits_{k\in\mathbb{Z}}a_{k}\biggl)^{p}=\bigwedge
\limits_{k\in\mathbb{Z}}a_{k}^{p}.
\]

\end{lemma}

\section{Kantorovich type max-min neural network operators}

Now, instead of $f\left(  \frac{k}{n}\right)  $ in \eqref{3}, taking the
average value of $f$ on $\left[  \frac{k}{n},\frac{k+1}{n}\right]  $, we
construct the Kantorovich type max-min neural network operator as follows%
\begin{equation}
K_{n}^{\left(  m\right)  }\left(  f;x\right)  :=\bigvee\limits_{k=\left\lceil
na\right\rceil }^{\left\lfloor nb\right\rfloor -1}n\int\limits_{\frac{k}{n}%
}^{\frac{k+1}{n}}f\left(  u\right)  du\wedge\dfrac{\phi_{\sigma}\left(
nx-k\right)  }{\bigvee\limits_{d=\left\lceil na\right\rceil }^{\left\lfloor
nb\right\rfloor -1}\phi_{\sigma}\left(  nx-d\right)  }\text{ \ \ (}x\in\left[
a,b\right]  \text{)}, \label{a}%
\end{equation}
where $f:\left[  a,b\right]  \rightarrow\left[  0,1\right]  $ is measurable
and $n\in\mathbb{N}$ is sufficiently large such that $\left\lceil
na\right\rceil \leq\left\lfloor nb\right\rfloor $. Lemma \ref{lemma2}
guarantees that the denominator in the operator is different from zero.
Therefore,
\begin{align*}
K_{n}^{\left(  m\right)  }\left(  f;x\right)   &  \leq\bigvee
\limits_{k=\left\lceil na\right\rceil }^{\left\lfloor nb\right\rfloor
-1}1\wedge\dfrac{\phi_{\sigma}\left(  nx-k\right)  }{\bigvee
\limits_{d=\left\lceil na\right\rceil }^{\left\lfloor nb\right\rfloor -1}%
\phi_{\sigma}\left(  nx-d\right)  }\\
&  =\bigvee\limits_{k=\left\lceil na\right\rceil }^{\left\lfloor
nb\right\rfloor -1}\dfrac{\phi_{\sigma}\left(  nx-k\right)  }{\bigvee
\limits_{d=\left\lceil na\right\rceil }^{\left\lfloor nb\right\rfloor -1}%
\phi_{\sigma}\left(  nx-d\right)  }\\
&  =1<\infty,
\end{align*}
that is, the Kantorovich type max-min operator is well-defined. Some important
properties of the max-min neural network operator are given in the following lemma.

\begin{lemma}
\label{lemma5} Let $f,g:\left[  a,b\right]  \rightarrow\left[  0,1\right]  $
be two measurable functions.

\begin{enumerate}
\item[(a)] If $\sigma$ is continuous function on $\mathbb{R}$, then
$K_{n}^{\left(  m\right)  }\left(  f\right)  $ is continuous on $\left[
a,b\right]  $,

\item[(b)] If $f\left(  x\right)  \leq g\left(  x\right)  $ for all
$x\in\left[  a,b\right]  ,$ then $K_{n}^{\left(  m\right)  }\left(
f;x\right)  \leq K_{n}^{\left(  m\right)  }\left(  g;x\right)  $ for all
$x\in\left[  a,b\right]  $,

\item[(c)] $K_{n}^{\left(  m\right)  }$ is sublinear, that is, $K_{n}^{\left(
m\right)  }\left(  f+g;x\right)  \leq K_{n}^{\left(  m\right)  }\left(
f;x\right)  +K_{n}^{\left(  m\right)  }\left(  g;x\right)  $ for all
$x\in\left[  a,b\right]  $,

\item[(d)] $\left\vert K_{n}^{\left(  m\right)  }\left(  f;x\right)
-K_{n}^{\left(  m\right)  }\left(  g;x\right)  \right\vert \leq K_{n}^{\left(
m\right)  }\left(  \left\vert f-g\right\vert ;x\right)  $ for all $x\in\left[
a,b\right]  $.
\end{enumerate}
\end{lemma}

The proof can be easily obtained from the definition of the operator and the
lemmas above.

\begin{remark}
\label{rem1}Kantorovich type max-min neural network operators are not
pseudo-linear in the max-min sense. Notice also that $K_{n}^{\left(  m\right)
}$ is not homogeneous, which means $K_{n}^{\left(  m\right)  }\left(
cf\right)  \neq cK_{n}^{\left(  m\right)  }\left(  f\right)  $ for some
measurable functions $f$ and constants $c>0.$
\end{remark}

For a fixed $\delta>0,$ $x\in\left[  a,b\right]  $ and $n\in\mathbb{N}$, we
define $B_{\delta,n}\left(  x\right)  $ such that%
\[
B_{\delta,n}\left(  x\right)  :=\left\{  k=\left\lceil na\right\rceil
,\left\lceil na\right\rceil +1,\ldots,\left\lfloor nb\right\rfloor
-1:\left\vert x-\frac{k}{n}\right\vert >\delta\right\}  \text{.}%
\]
Our approximation theorems now read as follows.

\begin{theorem}
\label{thm1}Let $f:\left[  a,b\right]  \rightarrow\left[  0,1\right]  $ be a
measurable function. Then we have
\[
\lim_{n\rightarrow\infty}K_{n}^{\left(  m\right)  }\left(  f;x\right)
=f\left(  x\right)
\]
at any continuity point $x\in\left[  a,b\right]  $ of $f.$ Furthermore, if
$f\in C\left(  \left[  a,b\right]  ,\left[  0,1\right]  \right)  ,$ we get%
\[
\lim_{n\rightarrow\infty}\left\Vert K_{n}^{\left(  m\right)  }\left(
f\right)  -f\right\Vert _{\infty}=0.
\]

\end{theorem}

\begin{proof}
Let $x\in\left[  a,b\right]  $ be a continuity point of $f$. Adding and
subtracting some suitable terms, by the triangle inequality we get%
\begin{align*}
&  \left\vert K_{n}^{\left(  m\right)  }\left(  f;x\right)  -f\left(
x\right)  \right\vert \\
&  \leq\bigg|\bigvee\limits_{k=\left\lceil na\right\rceil }^{\left\lfloor
nb\right\rfloor -1}n\int\limits_{\frac{k}{n}}^{\frac{k+1}{n}}f\left(
u\right)  du\wedge\dfrac{\phi_{\sigma}\left(  nx-k\right)  }{\bigvee
\limits_{d=\left\lceil na\right\rceil }^{\left\lfloor nb\right\rfloor -1}%
\phi_{\sigma}\left(  nx-d\right)  }\\
&  -\bigvee\limits_{k=\left\lceil na\right\rceil }^{\left\lfloor
nb\right\rfloor -1}n\int\limits_{\frac{k}{n}}^{\frac{k+1}{n}}f\left(
x\right)  du\wedge\dfrac{\phi_{\sigma}\left(  nx-k\right)  }{\bigvee
\limits_{d=\left\lceil na\right\rceil }^{\left\lfloor nb\right\rfloor -1}%
\phi_{\sigma}\left(  nx-d\right)  }\bigg|\\
&  +\bigg|\bigvee\limits_{k=\left\lceil na\right\rceil }^{\left\lfloor
nb\right\rfloor -1}n\int\limits_{\frac{k}{n}}^{\frac{k+1}{n}}f\left(
x\right)  du\wedge\dfrac{\phi_{\sigma}\left(  nx-k\right)  }{\bigvee
\limits_{d=\left\lceil na\right\rceil }^{\left\lfloor nb\right\rfloor -1}%
\phi_{\sigma}\left(  nx-d\right)  }-f\left(  x\right)  \bigg|
\end{align*}
where%
\begin{align*}
&  \bigg|\bigvee\limits_{k=\left\lceil na\right\rceil }^{\left\lfloor
nb\right\rfloor -1}n\int\limits_{\frac{k}{n}}^{\frac{k+1}{n}}f\left(
x\right)  du\wedge\dfrac{\phi_{\sigma}\left(  nx-k\right)  }{\bigvee
\limits_{d=\left\lceil na\right\rceil }^{\left\lfloor nb\right\rfloor -1}%
\phi_{\sigma}\left(  nx-d\right)  }-f\left(  x\right)  \bigg|\\
&  =\bigg|\bigvee\limits_{k=\left\lceil na\right\rceil }^{\left\lfloor
nb\right\rfloor -1}f\left(  x\right)  \wedge\dfrac{\phi_{\sigma}\left(
nx-k\right)  }{\bigvee\limits_{d=\left\lceil na\right\rceil }^{\left\lfloor
nb\right\rfloor -1}\phi_{\sigma}\left(  nx-d\right)  }-f\left(  x\right)
\bigg|\\
&  =\bigg|f\left(  x\right)  \wedge\bigvee\limits_{k=\left\lceil
na\right\rceil }^{\left\lfloor nb\right\rfloor -1}\dfrac{\phi_{\sigma}\left(
nx-k\right)  }{\bigvee\limits_{d=\left\lceil na\right\rceil }^{\left\lfloor
nb\right\rfloor -1}\phi_{\sigma}\left(  nx-d\right)  }-f\left(  x\right)
\bigg|\\
&  =0.
\end{align*}
Then from Lemma \ref{lemma3} and Lemma \ref{lemma4}, we obtain
\begin{align}
&  \left\vert K_{n}^{\left(  m\right)  }\left(  f;x\right)  -f\left(
x\right)  \right\vert \nonumber\\
&  \leq\bigvee\limits_{k=\left\lceil na\right\rceil }^{\left\lfloor
nb\right\rfloor -1}\bigg|n\int\limits_{\frac{k}{n}}^{\frac{k+1}{n}}f\left(
u\right)  du\wedge\dfrac{\phi_{\sigma}\left(  nx-k\right)  }{\bigvee
\limits_{d=\left\lceil na\right\rceil }^{\left\lfloor nb\right\rfloor -1}%
\phi_{\sigma}\left(  nx-d\right)  }-n\int\limits_{\frac{k}{n}}^{\frac{k+1}{n}%
}f\left(  x\right)  du\wedge\dfrac{\phi_{\sigma}\left(  nx-k\right)  }%
{\bigvee\limits_{d=\left\lceil na\right\rceil }^{\left\lfloor nb\right\rfloor
-1}\phi_{\sigma}\left(  nx-d\right)  }\bigg|\nonumber\\
&  \leq\bigvee\limits_{k=\left\lceil na\right\rceil }^{\left\lfloor
nb\right\rfloor -1}n\int\limits_{\frac{k}{n}}^{\frac{k+1}{n}}\left\vert
f\left(  u\right)  -f\left(  x\right)  \right\vert du\wedge\dfrac{\phi
_{\sigma}\left(  nx-k\right)  }{\bigvee\limits_{d=\left\lceil na\right\rceil
}^{\left\lfloor nb\right\rfloor -1}\phi_{\sigma}\left(  nx-d\right)  }.
\label{4}%
\end{align}
Since $f$ is continuous at the point $x,$ then for all $\varepsilon>0,$ there
exists a $\delta=\delta\left(  x,\varepsilon\right)  >0$ such that
\[
\left\vert f\left(  y\right)  -f\left(  x\right)  \right\vert <\varepsilon
\]
whenever $y\in\left[  x-\delta,x+\delta\right]  .$ Partitioning the maximum
operation as follows%
\begin{align*}
\left\vert K_{n}^{\left(  m\right)  }\left(  f;x\right)  -f\left(  x\right)
\right\vert  &  \leq\bigvee\limits_{k\in B_{\frac{\delta}{2},n}\left(
x\right)  }n\int\limits_{\frac{k}{n}}^{\frac{k+1}{n}}\left\vert f\left(
u\right)  -f\left(  x\right)  \right\vert du\wedge\dfrac{\phi_{\sigma}\left(
nx-k\right)  }{\bigvee\limits_{d=\left\lceil na\right\rceil }^{\left\lfloor
nb\right\rfloor -1}\phi_{\sigma}\left(  nx-d\right)  }\\
&  \bigvee\bigvee\limits_{k\notin B_{\frac{\delta}{2},n}\left(  x\right)
}n\int\limits_{\frac{k}{n}}^{\frac{k+1}{n}}\left\vert f\left(  u\right)
-f\left(  x\right)  \right\vert du\wedge\dfrac{\phi_{\sigma}\left(
nx-k\right)  }{\bigvee\limits_{d=\left\lceil na\right\rceil }^{\left\lfloor
nb\right\rfloor -1}\phi_{\sigma}\left(  nx-d\right)  }\\
&  =:U_{1}\bigvee U_{2}\text{,}%
\end{align*}
(using the fact that $\left\vert u-x\right\vert \leq\left\vert u-\frac{k}%
{n}\right\vert +\left\vert \frac{k}{n}-x\right\vert \leq\frac{1}{n}%
+\frac{\delta}{2}\leq\delta$ for sufficiently large $n\in\mathbb{N}$ in
$U_{2}$) we obtain%
\begin{align*}
U_{2}  &  <\bigvee\limits_{k\notin B_{\frac{\delta}{2},n}\left(  x\right)
}\varepsilon\wedge\dfrac{\phi_{\sigma}\left(  nx-k\right)  }{\bigvee
\limits_{d=\left\lceil na\right\rceil }^{\left\lfloor nb\right\rfloor -1}%
\phi_{\sigma}\left(  nx-d\right)  }\\
&  \leq\varepsilon.
\end{align*}
On the other hand, since $\left\vert f\right\vert \leq1,$ from Lemma
\ref{lemma0}\ we get
\begin{align*}
U_{1}  &  \leq\bigvee\limits_{k\in B_{\frac{\delta}{2},n}\left(  x\right)
}1\wedge\dfrac{\phi_{\sigma}\left(  nx-k\right)  }{\bigvee
\limits_{d=\left\lceil na\right\rceil }^{\left\lfloor nb\right\rfloor -1}%
\phi_{\sigma}\left(  nx-d\right)  }\\
&  \leq\bigvee\limits_{\left\vert nx-k\right\vert >n\frac{\delta}{2}}%
\dfrac{\phi_{\sigma}\left(  nx-k\right)  }{\phi_{\sigma}\left(  2\right)  }\\
&  \leq\frac{K}{\phi_{\sigma}\left(  2\right)  }\frac{1}{n^{1+\alpha}}\\
&  <\varepsilon.
\end{align*}
for sufficiently large $n\in\mathbb{N}$, which gives the desired result.

For the second part of the theorem, if $f\in C\left(  \left[  a,b\right]
,\left[  0,1\right]  \right)  $, then using a similar argumentation lines, and
noting that $\delta=\delta\left(  \varepsilon\right)  \,$, one can easily
complete the proof.
\end{proof}

\begin{theorem}
\label{theorem2}Let $f\in C\left(  \left[  a,b\right]  ,\left[  0,1\right]
\right)  .$ Then%
\[
\lim_{n\rightarrow\infty}\left\Vert K_{n}^{\left(  m\right)  }\left(
f\right)  -f\right\Vert _{p}=0,
\]
where $\left\Vert \cdot\right\Vert _{p}$ denotes the $L^{p}$ norm on $\left[
a,b\right]  \ $\ for $1\leq p<\infty$.
\end{theorem}

\begin{proof}
Since from the previous theorem%
\[
\left\Vert K_{n}^{\left(  m\right)  }\left(  f\right)  -f\right\Vert _{\infty
}<\varepsilon
\]
for sufficiently large $n\in\mathbb{N}$, then we get%

\begin{align*}
\left\Vert K_{n}^{\left(  m\right)  }\left(  f\right)  -f\right\Vert _{p}  &
=\biggl(\int\limits_{a}^{b}\left\vert K_{n}^{\left(  m\right)  }\left(
f;x\right)  -f\left(  x\right)  \right\vert ^{p}dx\biggl)^{\frac{1}{p}}\\
&  \leq\left\Vert K_{n}^{\left(  m\right)  }\left(  f\right)  -f\right\Vert
_{\infty}\left(  b-a\right)  ^{1/p}\\
&  <\varepsilon\left(  b-a\right)  ^{\frac{1}{p}}%
\end{align*}
for sufficiently large $n\in\mathbb{N}$, which completes the proof.
\end{proof}

\begin{theorem}
	\label{theorem4}Let $f\in L^{p}\left(  \left[  a,b\right]  ,\left[
	0,1\right]  \right)  $ for $1\leq p<\infty$. Then there holds
	\[
	\lim_{n\rightarrow\infty}\left\Vert K_{n}^{\left(  m\right)  }\left(
	f\right)  -f\right\Vert _{p}=0\text{.}%
	\]
	
\end{theorem}

\begin{proof}
	Assume that $f\in L^{p}\left(  \left[  a,b\right]  ,\left[  0,1\right]
	\right)  $ for $1\leq p<\infty.$ Since $C\left(  \left[  a,b\right]  ,\left[
	0,1\right]  \right)  $ is dense in $L^{p}\left(  \left[  a,b\right]  ,\left[
	0,1\right]  \right)  ,$ then for all $\varepsilon>0,$ there exists an $h\in
	C\left(  \left[  a,b\right]  ,\left[  0,1\right]  \right)  $ satisfying that%
	\[
	\left\Vert f-h\right\Vert _{p}<\varepsilon.
	\]
	We know that%
	\begin{align}
		\left\Vert K_{n}^{\left(  m\right)  }\left(  f\right)  -f\right\Vert _{p}  &
		\leq\left\Vert K_{n}^{\left(  m\right)  }\left(  f\right)  -K_{n}^{\left(
			m\right)  }\left(  h\right)  \right\Vert _{p}+\left\Vert K_{n}^{\left(
			m\right)  }\left(  h\right)  -h\right\Vert _{p}+\left\Vert h-f\right\Vert
		_{p}\label{5}\\
		&  <\left\Vert K_{n}^{\left(  m\right)  }\left(  f\right)  -K_{n}^{\left(
			m\right)  }\left(  h\right)  \right\Vert _{p}+\left\Vert K_{n}^{\left(
			m\right)  }\left(  h\right)  -h\right\Vert _{p}+\varepsilon.\nonumber
	\end{align}
	By Theorem \ref{theorem2}, we have
	\[
	\left\Vert K_{n}^{\left(  m\right)  }\left(  h\right)  -h\right\Vert
	_{p}<\varepsilon
	\]
	for sufficiently large $n\in\mathbb{N}.$ On the other hand, from Lemma
	\ref{lemma5} (d), we know that
	\begin{align*}
		&  \left\Vert K_{n}^{\left(  m\right)  }\left(  f\right)  -K_{n}^{\left(
			m\right)  }\left(  h\right)  \right\Vert _{p}\\
		&  =\biggl(\int\limits_{a}^{b}\left\vert K_{n}^{\left(  m\right)  }\left(
		f;x\right)  -K_{n}^{\left(  m\right)  }\left(  h;x\right)  \right\vert
		^{p}dx\biggl)^{\frac{1}{p}}\\
		&  \leq\biggl(\int\limits_{a}^{b}\biggl[\bigvee\limits_{k=\left\lceil
			na\right\rceil }^{\left\lfloor nb\right\rfloor -1}n\int\limits_{\frac{k}{n}%
		}^{\frac{k+1}{n}}\left\vert f\left(  u\right)  -h\left(  u\right)  \right\vert
		du\wedge\dfrac{\phi_{\sigma}\left(  nx-k\right)  }{\bigvee
			\limits_{d=\left\lceil na\right\rceil }^{\left\lfloor nb\right\rfloor -1}%
			\phi_{\sigma}\left(  nx-d\right)  }\biggr]^{p}dx\biggl)^{\frac{1}{p}}%
	\end{align*}
	holds true. Further, from Lemma \ref{lemma6}, we obtain%
	\begin{align*}
		&  \left\Vert K_{n}^{\left(  m\right)  }\left(  f\right)  -K_{n}^{\left(
			m\right)  }\left(  h\right)  \right\Vert _{p}\\
		&  \leq\biggl(\int\limits_{a}^{b}\bigvee\limits_{k=\left\lceil na\right\rceil
		}^{\left\lfloor nb\right\rfloor -1}\biggl(n\int\limits_{\frac{k}{n}}%
		^{\frac{k+1}{n}}\left\vert f\left(  u\right)  -h\left(  u\right)  \right\vert
		du\biggl)^{p}\wedge\biggl(\dfrac{\phi_{\sigma}\left(  nx-k\right)  }%
		{\bigvee\limits_{d=\left\lceil na\right\rceil }^{\left\lfloor nb\right\rfloor
				-1}\phi_{\sigma}\left(  nx-d\right)  }\biggl)^{p}dx\biggl)^{\frac{1}{p}}.
	\end{align*}
	Now, considering
	\[
	\dfrac{\phi_{\sigma}\left(  nx-k\right)  }{\bigvee\limits_{d=\left\lceil
			na\right\rceil }^{\left\lfloor nb\right\rfloor -1}\phi_{\sigma}\left(
		nx-d\right)  }\leq1\text{,}%
	\]
	then by the convexity of $\left\vert \cdot\right\vert ^{p}$ and Jensen's
	inequality, we have%
	\[
	\left\Vert K_{n}^{\left(  m\right)  }\left(  f\right)  -K_{n}^{\left(
		m\right)  }\left(  h\right)  \right\Vert _{p}\leq\biggl(\int\limits_{a}^{b}%
	\bigvee\limits_{k=\left\lceil na\right\rceil }^{\left\lfloor nb\right\rfloor
		-1}n\int\limits_{\frac{k}{n}}^{\frac{k+1}{n}}\left\vert f\left(  u\right)
	-h\left(  u\right)  \right\vert ^{p}du\wedge\dfrac{\phi_{\sigma}\left(
		nx-k\right)  }{\bigvee\limits_{d=\left\lceil na\right\rceil }^{\left\lfloor
			nb\right\rfloor -1}\phi_{\sigma}\left(  nx-d\right)  }dx\biggr)^{\frac{1}{p}}.
	\]
	{Now, for the latter part of the proof, we need the following estimate. By Lemma \ref{lemma1}
		$(4)$ there are $M,L>0$ such that $\phi_{\sigma}\left(  y\right)  \leq
		M\left\vert y\right\vert ^{-\left(  1+\alpha\right)  }$ for $\left\vert
		y\right\vert >L$, since $\phi_{\sigma}\leq1/2$ on $\mathbb{R}$, after replacing
		$M$ by $\max\left\{  M,L^{1+\alpha}/2\right\}  $ we may and do assume that
		\[
		\phi_{\sigma}\left(  y\right)  \leq\frac{M}{\left\vert y\right\vert
			^{1+\alpha}}\text{ \ \ for all }y\neq0\text{.}%
		\]
		Hence, for every $\lambda\in(0,1]$, if we split the integral at $\left\vert
		y\right\vert =\lambda^{-\frac{1}{1+\alpha}}$ and use that $\phi_{\sigma}$ is
		even, we get
		\begin{equation}
			\int\limits_{\mathbb{R}}\lambda\wedge\dfrac{\phi_{\sigma}\left(  y\right)
			}{\phi_{\sigma}\left(  2\right)  }dy\leq2\lambda^{-\frac{1}{1+\alpha}}%
			\lambda+\frac{2M}{\phi_{\sigma}\left(  2\right)  }\int\limits_{\lambda
				^{-\frac{1}{1+\alpha}}}^{\infty}\frac{dy}{y^{1+\alpha}}=\left\{  \frac
			{2M}{\alpha\phi_{\sigma}\left(  2\right)  }+2\right\}  \lambda^{\frac{\alpha
				}{1+\alpha}}\text{,}\label{aux}%
		\end{equation}
		which also holds for $\lambda=0$ since both sides of (\ref{aux}) vanish.
		
		Now, replacing the maximum by the sum of its terms and using
		Lemma \ref{lemma2} $(2)$, which holds for all $n\geq n_{0}$ where
		$n_{0}$ depends only on $\left[  a,b\right]  $ and $\sigma$, we get%
		\begin{align*}
			&  \left\Vert K_{n}^{\left(  m\right)  }\left(  f\right)  -K_{n}^{\left(
				m\right)  }\left(  h\right)  \right\Vert _{p}^{p}\\
			&  \leq\sum\limits_{k=\left\lceil na\right\rceil }^{\left\lfloor
				nb\right\rfloor -1}\int\limits_{a}^{b}\biggl\{n\int\limits_{\frac{k}{n}%
			}^{\frac{k+1}{n}}\left\vert f\left(  u\right)  -h\left(  u\right)  \right\vert
			^{p}du\wedge\dfrac{\phi_{\sigma}\left(  nx-k\right)  }{\phi_{\sigma}\left(
				2\right)  }\biggr\}dx\\
			&  \leq\sum\limits_{k=\left\lceil na\right\rceil }^{\left\lfloor
				nb\right\rfloor -1}\int\limits_{\mathbb{R}}\biggl\{n\int\limits_{\frac{k}{n}%
			}^{\frac{k+1}{n}}\left\vert f\left(  u\right)  -h\left(  u\right)  \right\vert
			^{p}du\wedge\dfrac{\phi_{\sigma}\left(  nx-k\right)  }{\phi_{\sigma}\left(
				2\right)  }\biggr\}dx\text{.}%
		\end{align*}
		After  $y=nx-k$ substitution in the above integral, we obtain
		\[
		\left\Vert K_{n}^{\left(  m\right)  }\left(  f\right)  -K_{n}^{\left(
			m\right)  }\left(  h\right)  \right\Vert _{p}^{p}\leq\frac{1}{n}%
		\sum\limits_{k=\left\lceil na\right\rceil }^{\left\lfloor nb\right\rfloor
			-1}\int\limits_{\mathbb{R}}\biggl\{n\int\limits_{\frac{k}{n}}^{\frac{k+1}{n}%
		}\left\vert f\left(  u\right)  -h\left(  u\right)  \right\vert ^{p}%
		du\wedge\dfrac{\phi_{\sigma}\left(  y\right)  }{\phi_{\sigma}\left(  2\right)
		}\biggr\}dy\text{.}%
		\]
		Because $\left\vert f-h\right\vert \leq1$, we have $n\int_{\frac{k}{n}}%
		^{\frac{k+1}{n}}\left\vert f\left(  u\right)  -h\left(  u\right)  \right\vert
		^{p}du\in\left[  0,1\right]  $ for every $k$, so from (\ref{aux}) 
		\[
		\left\Vert K_{n}^{\left(  m\right)  }\left(  f\right)  -K_{n}^{\left(
			m\right)  }\left(  h\right)  \right\Vert _{p}^{p}\leq\left\{  \frac{2M}%
		{\alpha\phi_{\sigma}\left(  2\right)  }+2\right\}  \frac{1}{n}\sum
		\limits_{k=\left\lceil na\right\rceil }^{\left\lfloor nb\right\rfloor
			-1}\biggl(n\int\limits_{\frac{k}{n}}^{\frac{k+1}{n}}\left\vert f\left(
		u\right)  -h\left(  u\right)  \right\vert ^{p}du\biggr)^{\frac{\alpha
			}{1+\alpha}}
		\]
		holds. Finally, the number of indices is $\left\lfloor nb\right\rfloor -\left\lceil
		na\right\rceil \leq n\left(  b-a\right)  $ and the intervals $\left[  \frac
		{k}{n},\frac{k+1}{n}\right]  $ are non-overlapping subintervals of $\left[
		a,b\right]  $, hence H\"{o}lder's inequality with the exponents $\frac
		{1+\alpha}{\alpha}$ and $1+\alpha$ gives%
		\begin{align*}
			&  \frac{1}{n}\sum\limits_{k=\left\lceil na\right\rceil }^{\left\lfloor
				nb\right\rfloor -1}\biggl(n\int\limits_{\frac{k}{n}}^{\frac{k+1}{n}}\left\vert
			f\left(  u\right)  -h\left(  u\right)  \right\vert ^{p}du\biggr)^{\frac{\alpha
				}{1+\alpha}}\\
			&  \leq\frac{1}{n}\left(  n\left(  b-a\right)  \right)  ^{\frac{1}{1+\alpha}%
			}\biggl(n\int\limits_{a}^{b}\left\vert f\left(  u\right)  -h\left(  u\right)
			\right\vert ^{p}du\biggr)^{\frac{\alpha}{1+\alpha}}=\left(  b-a\right)
			^{\frac{1}{1+\alpha}}\left\Vert f-h\right\Vert _{p}^{\frac{p\alpha}{1+\alpha}%
			}\text{.}%
		\end{align*}
		Collecting the above estimates we conclude that%
		\begin{align}
			\left\Vert K_{n}^{\left(  m\right)  }\left(  f\right)  -K_{n}^{\left(
				m\right)  }\left(  h\right)  \right\Vert _{p}  &  \leq\left\{  \frac{2M}%
			{\alpha\phi_{\sigma}\left(  2\right)  }+2\right\}  ^{\frac{1}{p}}\left(
			b-a\right)  ^{\frac{1}{p\left(  1+\alpha\right)  }}\left\Vert f-h\right\Vert
			_{p}^{\frac{\alpha}{1+\alpha}}\label{6}\\
			&  <\left\{  \frac{2M}{\alpha\phi_{\sigma}\left(  2\right)  }+2\right\}
			^{\frac{1}{p}}\left(  b-a\right)  ^{\frac{1}{p\left(  1+\alpha\right)  }%
			}\varepsilon^{\frac{\alpha}{1+\alpha}}\nonumber
		\end{align}
		holds for all $n\geq n_{0}$. Therefore, by the arbitrariness of
		$\varepsilon$, the proof is complete. We stress that both the constant and $n_{0}$ in (\ref{6})
		depend only on $\left[  a,b\right]  $, $\sigma$ and $p$, but not on $f$ and
		$h$. This uniformity is essential in the proof of Theorem \ref{thm6}, where the
		infimum over $h\in C^{1}\left(  \left[  a,b\right]  ,\left[  0,1\right]
		\right)  $ is taken for a fixed $n$.} 
\end{proof}
\begin{remark}
\label{rem3} Although we have shown the $L^{p}$-approximation only for
functions $f\in L^{p}\left(  \left[  a,b\right]  ,\left[  0,1\right]  \right)
$, it is possible to extend the results for functions $f\in L^{p}\left(
\left[  a,b\right]  ,\mathbb{R}\right)  $ by extending the definition of the
operator $K_{n}^{\left(  m\right)  }$ with respect to the range of the
functions in $L^{p}\left(  \left[  a,b\right]  ,\mathbb{R}\right)  $ ({see
Remark 3.1 in \cite{is}, see also \cite{gokcer2}}).
\end{remark}

\section{Estimates for the Kantorovich type max-min neural network operators}

Let $f:\left[  a,b\right]  \rightarrow\left[  0,1\right]  $ be given. Then for
a $\delta>0,$ the modulus of continuity of $f$ on $\left[  a,b\right]  $ is
defined by%
\[
\omega_{\left[  a,b\right]  }\left(  f,\delta\right)  :=\sup_{\left\vert
x-y\right\vert \leq\delta}\left\{  \left\vert f\left(  x\right)  -f\left(
y\right)  \right\vert :x,y\in\left[  a,b\right]  \right\}  \text{.}%
\]

We also need the definition of generalized absolute moment of order $\beta>0$,
introduced in \cite{cost2} such that for a given $\phi_{\sigma}$, it is
defined as follows,%
\[
m_{\beta}\left(  \phi_{\sigma}\right)  :=\sup_{x\in\mathbb{R}}\left\{
\bigvee\limits_{k\in\mathbb{Z}}\phi_{\sigma}\left(  x-k\right)  \left\vert
x-k\right\vert ^{\beta}\right\}  \text{.}%
\]

\begin{lemma}
\label{lemf} (see \cite{cost2}) If $0<\beta\leq1+\alpha$, then $m_{\beta}\left(  \phi_{\sigma
}\right)  <\infty$.
\end{lemma}

Now, we investigate the rate of approximation for Theorem \ref{thm1}.

\begin{theorem}
\label{thm5}Let $f\in C\left(  \left[  a,b\right]  ,\left[  0,1\right]
\right)  $ and $\delta_{n}$ be a null sequence of positive real numbers being
$\left(  n\delta_{n}\right)  ^{-1}$ a null sequence. Then, we have%
\[
\left\Vert K_{n}^{\left(  m\right)  }\left(  f\right)  -f\right\Vert _{\infty
}=\omega_{\left[  a,b\right]  }\left(  f,n^{-1}\right)  +\omega_{\left[
a,b\right]  }\left(  f,\delta_{n}\right)
{\displaystyle\bigvee}
\frac{m_{\left(  1+\alpha\right)  }\left(  \phi_{\sigma}\right)  }%
{\phi_{\sigma}\left(  2\right)  }\left(  n\delta_{n}\right)  ^{-\left(
1+\alpha\right)  }\text{.}%
\]

\end{theorem}

\begin{proof}
From (\ref{4}), we \ know that%
\[
\left\vert K_{n}^{\left(  m\right)  }\left(  f;x\right)  -f\left(  x\right)
\right\vert \leq\bigvee\limits_{k=\left\lceil na\right\rceil }^{\left\lfloor
nb\right\rfloor -1}n\int\limits_{\frac{k}{n}}^{\frac{k+1}{n}}\left\vert
f\left(  u\right)  -f\left(  x\right)  \right\vert du\wedge\dfrac{\phi
_{\sigma}\left(  nx-k\right)  }{\bigvee\limits_{d=\left\lceil na\right\rceil
}^{\left\lfloor nb\right\rfloor -1}\phi_{\sigma}\left(  nx-d\right)  }%
\]
holds. Since $\left\vert f\left(  u\right)  -f\left(  x\right)  \right\vert
\leq\left\vert f\left(  u\right)  -f\left(  \frac{k}{n}\right)  \right\vert
+\left\vert f\left(  \frac{k}{n}\right)  -f\left(  x\right)  \right\vert $ and
$\left(  a+b\right)  \wedge c\leq a\wedge c+b\wedge c$ for all $a,b,c\geq0$,
the following inequality holds true%
\begin{align}
\left\vert K_{n}^{\left(  m\right)  }\left(  f;x\right)  -f\left(  x\right)
\right\vert  &  \leq\bigvee\limits_{k=\left\lceil na\right\rceil
}^{\left\lfloor nb\right\rfloor -1}n\int\limits_{\frac{k}{n}}^{\frac{k+1}{n}%
}\left\vert f\left(  u\right)  -f\left(  \frac{k}{n}\right)  \right\vert
du\wedge\dfrac{\phi_{\sigma}\left(  nx-k\right)  }{\bigvee
\limits_{d=\left\lceil na\right\rceil }^{\left\lfloor nb\right\rfloor -1}%
\phi_{\sigma}\left(  nx-d\right)  }\nonumber\\
&  +\bigvee\limits_{k=\left\lceil na\right\rceil }^{\left\lfloor
nb\right\rfloor -1}n\int\limits_{\frac{k}{n}}^{\frac{k+1}{n}}\left\vert
f\left(  \frac{k}{n}\right)  -f\left(  x\right)  \right\vert du\wedge
\dfrac{\phi_{\sigma}\left(  nx-k\right)  }{\bigvee\limits_{d=\left\lceil
na\right\rceil }^{\left\lfloor nb\right\rfloor -1}\phi_{\sigma}\left(
nx-d\right)  }\label{55}\\
&  =:I_{1}+I_{2}\text{,}\nonumber
\end{align}
where the integral in $I_{2}$ does not depend on $u$ and hence
\[
I_{2}=\bigvee\limits_{k=\left\lceil na\right\rceil }^{\left\lfloor
nb\right\rfloor -1}\left\vert f\left(  \frac{k}{n}\right)  -f\left(  x\right)
\right\vert \wedge\dfrac{\phi_{\sigma}\left(  nx-k\right)  }{\bigvee
\limits_{d=\left\lceil na\right\rceil }^{\left\lfloor nb\right\rfloor -1}%
\phi_{\sigma}\left(  nx-d\right)  }.
\]
If we divide $I_{2}$ as given below%
\begin{align*}
I_{2}  &  =\bigvee\limits_{k\in B_{\delta_{n},n}\left(  x\right)  }\left\vert
f\left(  \frac{k}{n}\right)  -f\left(  x\right)  \right\vert \wedge\dfrac
{\phi_{\sigma}\left(  nx-k\right)  }{\bigvee\limits_{d=\left\lceil
na\right\rceil }^{\left\lfloor nb\right\rfloor -1}\phi_{\sigma}\left(
nx-d\right)  }\\
&  \bigvee\bigvee\limits_{k\notin B_{\delta_{n},n}\left(  x\right)
}\left\vert f\left(  \frac{k}{n}\right)  -f\left(  x\right)  \right\vert
\wedge\dfrac{\phi_{\sigma}\left(  nx-k\right)  }{\bigvee\limits_{d=\left\lceil
na\right\rceil }^{\left\lfloor nb\right\rfloor -1}\phi_{\sigma}\left(
nx-d\right)  }\\
&  =:I_{2}^{1}%
{\textstyle\bigvee}
I_{2}^{2}%
\end{align*}
then%
\begin{align*}
I_{2}^{2}  &  \leq\bigvee\limits_{k\notin B_{\delta_{n},n}\left(  x\right)
}\omega_{\left[  a,b\right]  }\left(  f,\delta_{n}\right)  \wedge\dfrac
{\phi_{\sigma}\left(  nx-k\right)  }{\bigvee\limits_{d=\left\lceil
na\right\rceil }^{\left\lfloor nb\right\rfloor -1}\phi_{\sigma}\left(
nx-d\right)  }\\
&  \leq\omega_{\left[  a,b\right]  }\left(  f,\delta_{n}\right)
\end{align*}
holds. In $I_{2}^{1}$ since $k\in B_{\delta_{n},n}\left(  x\right)  ,$ we get
$\frac{\left\vert nx-k\right\vert ^{1+\alpha}}{\left(  n\delta_{n}\right)
^{1+\alpha}}>1$ and therefore%
\begin{align*}
I_{2}^{1}  &  \leq\bigvee\limits_{k\in B_{\delta_{n},n}\left(  x\right)
}\left\vert f\left(  \frac{k}{n}\right)  -f\left(  x\right)  \right\vert
\wedge\dfrac{\phi_{\sigma}\left(  nx-k\right)  }{\bigvee\limits_{d=\left\lceil
na\right\rceil }^{\left\lfloor nb\right\rfloor -1}\phi_{\sigma}\left(
nx-d\right)  }\\
&  \leq\bigvee\limits_{k\in B_{\delta_{n},n}\left(  x\right)  }\dfrac
{\phi_{\sigma}\left(  nx-k\right)  }{\phi_{\sigma}\left(  2\right)  }\\
&  \leq\bigvee\limits_{k\in B_{\delta_{n},n}\left(  x\right)  }\dfrac
{\phi_{\sigma}\left(  nx-k\right)  }{\phi_{\sigma}\left(  2\right)  }%
\frac{\left\vert nx-k\right\vert ^{1+\alpha}}{\left(  n\delta_{n}\right)
^{1+\alpha}}\\
&  \leq\frac{m_{\left(  1+\alpha\right)  }\left(  \phi_{\sigma}\right)  }%
{\phi_{\sigma}\left(  2\right)  }\frac{1}{\left(  n\delta_{n}\right)
^{1+\alpha}}%
\end{align*}
holds, where $m_{\left(  1+\alpha\right)  }\left(  \phi_{\sigma}\right)  $ is
finite from Lemma \ref{lemf}. Now using the well known properties of modulus
of continutiy in $I_{1}$, we obtain%
\begin{align*}
I_{1}  &  \leq\bigvee\limits_{k=\left\lceil na\right\rceil }^{\left\lfloor
nb\right\rfloor -1}n\int\limits_{\frac{k}{n}}^{\frac{k+1}{n}}\omega_{\left[
a,b\right]  }\left(  f,\left\vert u-\frac{k}{n}\right\vert \right)
du\wedge\dfrac{\phi_{\sigma}\left(  nx-k\right)  }{\bigvee
\limits_{d=\left\lceil na\right\rceil }^{\left\lfloor nb\right\rfloor -1}%
\phi_{\sigma}\left(  nx-d\right)  }\\
&  \leq\bigvee\limits_{k=\left\lceil na\right\rceil }^{\left\lfloor
nb\right\rfloor -1}\omega_{\left[  a,b\right]  }\left(  f,\frac{1}{n}\right)
\wedge\dfrac{\phi_{\sigma}\left(  nx-k\right)  }{\bigvee\limits_{d=\left\lceil
na\right\rceil }^{\left\lfloor nb\right\rfloor -1}\phi_{\sigma}\left(
nx-d\right)  }\\
&  =\omega_{\left[  a,b\right]  }\left(  f,\frac{1}{n}\right)  \text{,}%
\end{align*}

which completes the proof.
\end{proof}

In the above theory, if we consider the H\"{o}lder continuous functions of
order $\beta$, that is, for a given $\beta\in(0,1],$ $Lip_{\left[  a,b\right]
}\left(  \beta\right)  $ is defined by%
\begin{align*}
Lip_{\left[  a,b\right]  }\left(  \beta\right)   &  :=\{ f\in C\left(  \left[
a,b\right]  ,\left[  0,1\right]  \right)  :\exists K>0\text{ such that
}\left\vert f\left(  x\right)  -f\left(  y\right)  \right\vert \leq
K\left\vert x-y\right\vert ^{\beta}\\
&  \left.  \text{ \ \ \ \ \ for all~}x,y\in\left[  a,b\right]  \right\}
\text{,}%
\end{align*}
then we get the following rate of approximation.

\begin{corollary}
\label{corol1}Let $f\in Lip_{\left[  a,b\right]  }\left(  \beta\right)  $.
Then
\[
\left\Vert K_{n}^{\left(  m\right)  }\left(  f\right)  -f\right\Vert _{\infty
}=O\left(  n^{-\frac{\left(  1+\alpha\right)  \beta}{1+\alpha+\beta}}\right)
\text{ as }n\rightarrow\infty
\]
holds.
\end{corollary}

\begin{proof}
Since $f\in Lip_{\left[  a,b\right]  }\left(  \beta\right)  $, from (\ref{55})
we can easily see that%
\begin{align*}
\left\vert K_{n}^{\left(  m\right)  }\left(  f;x\right)  -f\left(  x\right)
\right\vert  &  \leq\bigvee\limits_{k=\left\lceil na\right\rceil
}^{\left\lfloor nb\right\rfloor -1}Kn\int\limits_{\frac{k}{n}}^{\frac{k+1}{n}%
}\left(  u-\frac{k}{n}\right)  ^{\beta}du\wedge\dfrac{\phi_{\sigma}\left(
nx-k\right)  }{\bigvee\limits_{d=\left\lceil na\right\rceil }^{\left\lfloor
nb\right\rfloor -1}\phi_{\sigma}\left(  nx-d\right)  }\\
&  +\bigvee\limits_{k=\left\lceil na\right\rceil }^{\left\lfloor
nb\right\rfloor -1}Kn\int\limits_{\frac{k}{n}}^{\frac{k+1}{n}}\left\vert
\frac{k}{n}-x\right\vert ^{\beta}du\wedge\dfrac{\phi_{\sigma}\left(
nx-k\right)  }{\bigvee\limits_{d=\left\lceil na\right\rceil }^{\left\lfloor
nb\right\rfloor -1}\phi_{\sigma}\left(  nx-d\right)  }\\
&  =:J_{1}+J_{2}%
\end{align*}
where
\[
J_{1}=\frac{K}{\left(  \beta+1\right)  n^{\beta}}%
\]
for some $K>0$. Taking $\delta_{n}=\dfrac{1}{n^{\left(  1+\alpha\right)
/\left(  1+\alpha+\beta\right)  }}$, if we seperate the maximum operation in
$J_{2}$ as follows,
\begin{align*}
J_{2}  &  =\bigvee\limits_{k\in B_{\delta_{n},n}\left(  x\right)  }%
Kn\int\limits_{\frac{k}{n}}^{\frac{k+1}{n}}\left\vert \frac{k}{n}-x\right\vert
^{\beta}du\wedge\dfrac{\phi_{\sigma}\left(  nx-k\right)  }{\bigvee
\limits_{d=\left\lceil na\right\rceil }^{\left\lfloor nb\right\rfloor -1}%
\phi_{\sigma}\left(  nx-d\right)  }\\
&
{\displaystyle\bigvee}
\bigvee\limits_{k\notin B_{\delta_{n},n}\left(  x\right)  }Kn\int
\limits_{\frac{k}{n}}^{\frac{k+1}{n}}\left\vert \frac{k}{n}-x\right\vert
^{\beta}du\wedge\dfrac{\phi_{\sigma}\left(  nx-k\right)  }{\bigvee
\limits_{d=\left\lceil na\right\rceil }^{\left\lfloor nb\right\rfloor -1}%
\phi_{\sigma}\left(  nx-d\right)  }\\
&  =:J_{2}^{1}%
{\textstyle\bigvee}
J_{2}^{2}%
\end{align*}

then there holds%
\begin{align}
J_{2}^{1}  &  \leq\bigvee\limits_{k\in B_{\delta_{n},n}\left(  x\right)
}\dfrac{\phi_{\sigma}\left(  nx-k\right)  }{\bigvee\limits_{d=\left\lceil
na\right\rceil }^{\left\lfloor nb\right\rfloor -1}\phi_{\sigma}\left(
nx-d\right)  }\label{b}\\
&  \leq\bigvee\limits_{k\in B_{\delta_{n},n}\left(  x\right)  }\dfrac
{\phi_{\sigma}\left(  nx-k\right)  }{\phi_{\sigma}\left(  2\right)  }%
\frac{\left\vert nx-k\right\vert ^{1+\alpha}}{\left(  n\delta_{n}\right)
^{1+\alpha}}\nonumber\\
&  \leq\frac{m_{1+\alpha}\left(  \phi_{\sigma}\right)  }{\phi_{\sigma}\left(
2\right)  }\frac{1}{n^{\frac{\beta\left(  1+\alpha\right)  }{1+\alpha+\beta}}%
}\text{.}\nonumber
\end{align}
On the other hand, since $k\notin B_{\delta_{n},n}\left(  x\right)  $ in
$J_{2}^{2}$ we have
\begin{align}
J_{2}^{2}  &  \leq\bigvee\limits_{k\notin B_{\delta_{n},n}\left(  x\right)
}K\left(  \delta_{n}\right)  ^{\beta}\wedge\dfrac{\phi_{\sigma}\left(
nx-k\right)  }{\bigvee\limits_{d=\left\lceil na\right\rceil }^{\left\lfloor
nb\right\rfloor -1}\phi_{\sigma}\left(  nx-d\right)  }\label{c}\\
&  \leq\frac{K}{n^{\frac{\beta\left(  1+\alpha\right)  }{1+\alpha+\beta}}%
}\text{.}\nonumber
\end{align}
Finally, since
\begin{equation}
\frac{K}{\left(  \beta+1\right)  n^{\beta}}\leq\frac{1}{n^{\frac{\beta\left(
1+\alpha\right)  }{1+\alpha+\beta}}} \label{d}%
\end{equation}
for sufficiently large $n\in\mathbb{N}$, from (\ref{b}), (\ref{c}) and
(\ref{d}) we conclude%
\[
\left\Vert K_{n}^{\left(  m\right)  }\left(  f\right)  -f\right\Vert _{\infty
}=O\left(  n^{-\frac{\left(  1+\alpha\right)  \beta}{1+\alpha+\beta}}\right)
\text{ as }n\rightarrow\infty\text{.}%
\]

\end{proof}

In this part, inspired by \cite{est3,est1,est2}, we provide quantitative
estimates for Kantorovich type max-min neural network operators with the help
of $K$-functionals introduced by Peetre in \cite{peetre}. First we recall the
definition of $K$-functionals, which is adapted to max-min case. For a given
$f\in L^{p}\left(  \left[  a,b\right]  ,\left[  0,1\right]  \right)  $ $(1\leq
p<\infty)$%

\[
\mathcal{K}\left(  f,\delta\right)  _{p}:=\inf_{g\in C^{1}\left(  \left[
a,b\right]  ,\left[  0,1\right]  \right)  }\left\{  \left\Vert f-g\right\Vert
_{p}^{\frac{\alpha}{\alpha+1}}+\delta\left\Vert g^{\prime}\right\Vert
_{\infty}\right\}
\]
where $\delta>0$. According to this definition, $\mathcal{K}\left(
f,\delta\right)  _{p}<\varepsilon$ $\left(  \varepsilon>0\right)  $ for
sufficiently small $\delta>0$ means that $f$ can be approximated with an error
$\left\Vert f-g\right\Vert _{p}<\varepsilon^{\frac{\alpha+1}{\alpha}},$ where
$g\in C^{1}\left(  \left[  a,b\right]  ,\left[  0,1\right]  \right)  \subset
L^{p}\left(  \left[  a,b\right]  ,\left[  0,1\right]  \right)  $ and whose
derivative under supremum norm is not too large. $K$-functionals provide us
some information about the smoothness and approximation properties of $f$, and
may be considered as a modulus of smoothness in some situations in $L^{p}$
spaces. For the importance of $K$-functionals in approximation theory, we
refer to \cite{butzer}.

\begin{theorem}
	\label{thm6}Let $f\in L^{p}\left(  \left[  a,b\right]  ,\left[  0,1\right]
	\right)  $ for $1\leq p<\infty$. Then we have
	\[
	\left\Vert K_{n}^{\left(  m\right)  }\left(  f\right)  -f\right\Vert _{p}\leq
	A\mathcal{K}\left(  f,Bn^{-\frac{1+\alpha}{2+\alpha}}\right)  _{p}%
	+\frac{m_{\left(  1+\alpha\right)  }\left(  \phi_{\sigma}\right)  \left(
		b-a\right)  ^{\frac{1}{p}}}{\phi_{\sigma}\left(  2\right)  }n^{-\frac
		{1+\alpha}{2+\alpha}}%
	\]
	{for all $n\geq n_{0}$, where $n_{0}$ is the threshold in (\ref{6}) and $A=\left(  b-a\right)  ^{\frac{1}{p\left(  1+\alpha\right)  }}\left[  \left\{
		\frac{2M}{\alpha\phi_{\sigma}\left(  2\right)  }+2\right\}  ^{\frac{1}{p}%
		}+1\right]$ and $B=\frac{\left(  3/2\right)  \left(
			b-a\right)  ^{\frac{1}{p}}}{A}$. Here $M$
		denotes the constant introduced in the proof of Theorem \ref{theorem4}.}
\end{theorem}

\begin{proof}
	Let $h\in C^{1}\left(  \left[  a,b\right]  ,\left[  0,1\right]  \right)  $ be
	given. We know from (\ref{5}) and (\ref{6}) that%
	\begin{align*}
		&  \left\Vert K_{n}^{\left(  m\right)  }\left(  f\right)  -f\right\Vert _{p}\\
		&  \leq\left\Vert K_{n}^{\left(  m\right)  }\left(  f\right)  -K_{n}^{\left(
			m\right)  }\left(  h\right)  \right\Vert _{p}+\left\Vert K_{n}^{\left(
			m\right)  }\left(  h\right)  -h\right\Vert _{p}+\left\Vert h-f\right\Vert
		_{p}\\
		& {\leq\left(  \left\{  \frac{2M}{\alpha\phi_{\sigma}\left(  2\right)
			}+2\right\}  ^{\frac{1}{p}}\left(  b-a\right)  ^{\frac{1}{p\left(
					1+\alpha\right)  }}+\left\Vert f-h\right\Vert _{p}^{\frac{1}{1+\alpha}%
			}\right)  \left\Vert f-h\right\Vert _{p}^{\frac{\alpha}{1+\alpha}}+\left\Vert
			K_{n}^{\left(  m\right)  }\left(  h\right)  -h\right\Vert _{p}}\\
		&  {\leq\left(  b-a\right)  ^{\frac{1}{p\left(  1+\alpha\right)  }}\left(
			\left\{  \frac{2M}{\alpha\phi_{\sigma}\left(  2\right)  }+2\right\}
			^{\frac{1}{p}}+1\right)  \left\Vert f-h\right\Vert _{p}^{\frac{\alpha
				}{1+\alpha}}+\left\Vert K_{n}^{\left(  m\right)  }\left(  h\right)
			-h\right\Vert _{p}\text{.}}
	\end{align*}
	On the other hand, since $h\in C^{1}\left(  \left[  a,b\right]  ,\left[
	0,1\right]  \right)  $%
	\begin{align*}
		\left\vert K_{n}^{\left(  m\right)  }\left(  h;x\right)  -h\left(  x\right)
		\right\vert  &  \leq\bigvee\limits_{k=\left\lceil na\right\rceil
		}^{\left\lfloor nb\right\rfloor -1}n\int\limits_{\frac{k}{n}}^{\frac{k+1}{n}%
		}\left\vert h\left(  u\right)  -h\left(  x\right)  \right\vert du\wedge
		\dfrac{\phi_{\sigma}\left(  nx-k\right)  }{\bigvee\limits_{d=\left\lceil
				na\right\rceil }^{\left\lfloor nb\right\rfloor -1}\phi_{\sigma}\left(
			nx-d\right)  }\\
		&  \leq\bigvee\limits_{k=\left\lceil na\right\rceil }^{\left\lfloor
			nb\right\rfloor -1}n\left\Vert h^{\prime}\right\Vert _{\infty}\int
		\limits_{\frac{k}{n}}^{\frac{k+1}{n}}\left\vert x-u\right\vert du\wedge
		\dfrac{\phi_{\sigma}\left(  nx-k\right)  }{\bigvee\limits_{d=\left\lceil
				na\right\rceil }^{\left\lfloor nb\right\rfloor -1}\phi_{\sigma}\left(
			nx-d\right)  }\\
		&  \leq\bigvee\limits_{k=\left\lceil na\right\rceil }^{\left\lfloor
			nb\right\rfloor -1}n\left\Vert h^{\prime}\right\Vert _{\infty}\int
		\limits_{\frac{k}{n}}^{\frac{k+1}{n}}\left(  u-\frac{k}{n}\right)
		du\wedge\dfrac{\phi_{\sigma}\left(  nx-k\right)  }{\bigvee
			\limits_{d=\left\lceil na\right\rceil }^{\left\lfloor nb\right\rfloor -1}%
			\phi_{\sigma}\left(  nx-d\right)  }\\
		&  +\bigvee\limits_{k=\left\lceil na\right\rceil }^{\left\lfloor
			nb\right\rfloor -1}n\left\Vert h^{\prime}\right\Vert _{\infty}\int
		\limits_{\frac{k}{n}}^{\frac{k+1}{n}}\left\vert x- \frac{k}{n}\right\vert
		du\wedge\dfrac{\phi_{\sigma}\left(  nx-k\right)  }{\bigvee
			\limits_{d=\left\lceil na\right\rceil }^{\left\lfloor nb\right\rfloor -1}%
			\phi_{\sigma}\left(  nx-d\right)  }\\
		&  =:V_{1}+V_{2}%
	\end{align*}
	holds. Now evaluating the integral in $V_{1}$, we obtain that%
	\begin{align*}
		V_{1}  &  =\bigvee\limits_{k=\left\lceil na\right\rceil }^{\left\lfloor
			nb\right\rfloor -1}n\left\Vert h^{\prime}\right\Vert _{\infty}\int
		\limits_{\frac{k}{n}}^{\frac{k+1}{n}}\left(  u-\frac{k}{n}\right)
		du\wedge\dfrac{\phi_{\sigma}\left(  nx-k\right)  }{\bigvee
			\limits_{d=\left\lceil na\right\rceil }^{\left\lfloor nb\right\rfloor -1}%
			\phi_{\sigma}\left(  nx-d\right)  }\\
		&  =\bigvee\limits_{k=\left\lceil na\right\rceil }^{\left\lfloor
			nb\right\rfloor -1}\frac{\left\Vert h^{\prime}\right\Vert _{\infty}}{2n}%
		\wedge\dfrac{\phi_{\sigma}\left(  nx-k\right)  }{\bigvee\limits_{d=\left\lceil
				na\right\rceil }^{\left\lfloor nb\right\rfloor -1}\phi_{\sigma}\left(
			nx-d\right)  }\\
		&  \leq\frac{\left\Vert h^{\prime}\right\Vert _{\infty}}{2n}\leq\frac{\left\Vert
			h^{\prime}\right\Vert _{\infty}}{2n^{\frac{1+\alpha}{2+\alpha}}}\text{.}%
	\end{align*}
	It is obvious that
	\begin{align*}
		V_{2}  &  \leq\bigvee\limits_{k=\left\lceil na\right\rceil }^{\left\lfloor
			nb\right\rfloor -1}\left\Vert h^{\prime}\right\Vert _{\infty}\left\vert
		x-\frac{k}{n}\right\vert \wedge\dfrac{\phi_{\sigma}\left(  nx-k\right)
		}{\bigvee\limits_{d=\left\lceil na\right\rceil }^{\left\lfloor nb\right\rfloor
				-1}\phi_{\sigma}\left(  nx-d\right)  }\text{.}
	\end{align*}
	Now, separating the maximum operation as follows
	\begin{align*}
		V_{2}  &  \leq\bigvee\limits_{k\in B_{\delta_{n},n}\left(  x\right)
		}\left\Vert h^{\prime}\right\Vert _{\infty}\left\vert x-\frac{k}{n}\right\vert
		\wedge\dfrac{\phi_{\sigma}\left(  nx-k\right)  }{\bigvee\limits_{d=\left\lceil
				na\right\rceil }^{\left\lfloor nb\right\rfloor -1}\phi_{\sigma}\left(
			nx-d\right)  }\\
		&
		{\displaystyle\bigvee}
		\bigvee\limits_{k\notin B_{\delta_{n},n}\left(  x\right)  }\left\Vert
		h^{\prime}\right\Vert _{\infty}\left\vert x-\frac{k}{n}\right\vert
		\wedge\dfrac{\phi_{\sigma}\left(  nx-k\right)  }{\bigvee\limits_{d=\left\lceil
				na\right\rceil }^{\left\lfloor nb\right\rfloor -1}\phi_{\sigma}\left(
			nx-d\right)  }\\
		&  =:V_{2}^{1}%
		{\displaystyle\bigvee}
		V_{2}^{2}%
	\end{align*}
	where $\delta_{n}=\frac{1}{n^{\frac{1+\alpha}{2+\alpha}}}$ we obtain%
	\[
	V_{2}^{2}\leq\frac{\left\Vert h^{\prime}\right\Vert _{\infty}}{n^{\frac
			{1+\alpha}{2+\alpha}}}\text{.}%
	\]
	In $V_{2}^{1}$, since $k\in B_{\delta_{n},n}\left(  x\right)  $,
	\begin{align*}
		V_{2}^{1}  &  \leq\bigvee\limits_{k\in B_{\delta_{n},n}\left(  x\right)
		}\dfrac{\phi_{\sigma}\left(  nx-k\right)  }{\phi_{\sigma}\left(  2\right)  }\\
		&  \leq\bigvee\limits_{k\in B_{\delta_{n},n}\left(  x\right)  }\dfrac
		{\phi_{\sigma}\left(  nx-k\right)  }{\phi_{\sigma}\left(  2\right)  }%
		\frac{\left\vert nx-k\right\vert ^{1+\alpha}}{\left(  n\delta_{n}\right)
			^{1+\alpha}}\\
		&  \leq\frac{m_{\left(  1+\alpha\right)  }\left(  \phi_{\sigma}\right)  }%
		{\phi_{\sigma}\left(  2\right)  }\frac{1}{n^{\frac{1+\alpha}{2+\alpha}}}%
	\end{align*}
	and therefore
	\begin{align*}
		V_{2}  &  \leq\frac{\left\Vert h^{\prime}\right\Vert _{\infty}}{n^{\frac
				{1+\alpha}{2+\alpha}}}%
		{\displaystyle\bigvee}
		\frac{m_{\left(  1+\alpha\right)  }\left(  \phi_{\sigma}\right)  }%
		{\phi_{\sigma}\left(  2\right)  }\frac{1}{n^{\frac{1+\alpha}{2+\alpha}}}\\
		&  \leq\frac{\left\Vert h^{\prime}\right\Vert _{\infty}}{n^{\frac{1+\alpha
				}{2+\alpha}}}+\frac{m_{\left(  1+\alpha\right)  }\left(  \phi_{\sigma}\right)
		}{\phi_{\sigma}\left(  2\right)  }\frac{1}{n^{\frac{1+\alpha}{2+\alpha}}}%
	\end{align*}
	holds. Then we obtain%
	\[
	\left\Vert K_{n}^{\left(  m\right)  }\left(  h\right)  -h\right\Vert _{p}%
	\leq\left(  \frac{3}{2n^{\frac{1+\alpha}{2+\alpha}}}\left\Vert h^{\prime
	}\right\Vert _{\infty}+\frac{m_{\left(  1+\alpha\right)  }\left(  \phi
		_{\sigma}\right)  }{\phi_{\sigma}\left(  2\right)  }\frac{1}{n^{\frac
			{1+\alpha}{2+\alpha}}}\right)  \left(  b-a\right)  ^{\frac{1}{p}}%
	\]
	and hence%
	{\begin{align*}
			\left\Vert K_{n}^{\left(  m\right)  }\left(  f\right)  -f\right\Vert _{p}  &\leq\left(  b-a\right)  ^{\frac{1}{p\left(  1+\alpha\right)  }}\left(
			\left\{  \frac{2M}{\alpha\phi_{\sigma}\left(  2\right)  }+2\right\}
			^{\frac{1}{p}}+1\right)  \left\Vert f-h\right\Vert _{p}^{\frac{\alpha
				}{1+\alpha}}\\
			&  +\frac{3\left(  b-a\right)  ^{\frac{1}{p}}}{2}\frac{1}{n^{\frac{1+\alpha
					}{2+\alpha}}}\left\Vert h^{\prime}\right\Vert _{\infty}+\frac{m_{\left(
					1+\alpha\right)  }\left(  \phi_{\sigma}\right)  \left(  b-a\right)  ^{\frac
					{1}{p}}}{\phi_{\sigma}\left(  2\right)  }\frac{1}{n^{\frac{1+\alpha}{2+\alpha
			}}}\\
			& \leq A\left(  \left\Vert f-h\right\Vert _{p}^{\frac{\alpha}{1+\alpha}%
			}+B\frac{\left\Vert h^{\prime}\right\Vert _{\infty}}{n^{\frac{1+\alpha
					}{2+\alpha}}}\right)  +\frac{m_{\left(  1+\alpha\right)  }\left(  \phi
				_{\sigma}\right)  \left(  b-a\right)  ^{\frac{1}{p}}}{\phi_{\sigma}\left(
				2\right)  }\frac{1}{n^{\frac{1+\alpha}{2+\alpha}}}\text{.}%
		\end{align*}
		Since $h\in C^{1}\left(  \left[  a,b\right]  ,\left[  0,1\right]  \right)  $
		was arbitrary, passing to the infimum over all such $h$ we obtain}
	\[
	\left\Vert K_{n}^{\left(  m\right)  }\left(  f\right)  -f\right\Vert _{p}\leq
	A\mathcal{K}\left(  f,Bn^{-\frac{1+\alpha}{2+\alpha}}\right)  _{p}%
	+\frac{m_{\left(  1+\alpha\right)  }\left(  \phi_{\sigma}\right)  \left(
		b-a\right)  ^{\frac{1}{p}}}{\phi_{\sigma}\left(  2\right)  }\frac{1}%
	{n^{\frac{1+\alpha}{2+\alpha}}}\text{,}%
	\]
	which completes the proof.
\end{proof}

\section{Applications}

We present now some well-known sigmoidal functions which satisfy the
conditions of our theory and consider a few scenarios where the Kantorovich
type max-min neural network operators can be applied. In particular, we study
the influence of the sigmoidal function $\sigma$ and the parameter $n$ on the
approximation of discontinuous functions and the effects of these operator
when applied to noisy signals, for instance, in the case of noisy ECG signals.
Furthermore, we compare the approximation error of this operator with those of
other well-known neural-network operators as the max-product operators and the
classical linear variants in the space $L^{1}\left(  \left[  0,1\right]
,\left[  0,1\right]  \right)  $.

\vspace{2mm}

\textbf{\noindent Sigmoidal functions.} Regarding admissible sigmoidal
functions, it is well-known that the logistic sigmoidal function
(\cite{anas1,anas3,anas2,cao})
\[
\sigma_{l}:=\frac{1}{1+e^{-x}}\text{ \ \ }\left(  x\in\mathbb{R}\right)
\]
and the hyperbolic tangent function (\cite{anas3, anas4, anas5})%
\[
\sigma_{h}:=\frac{1}{2}\left(  \tanh x+1\right)  \text{ \ \ }\left(
x\in\mathbb{R}\right)
\]
satisfy condition $\left(  \Sigma3\right)  $ for all $\alpha>0.$ In view of
Remark \ref{rem0}, also the ramp function (\cite{cao2, cheang})%
\[
\sigma_{R}:=\left\{
\begin{array}
[c]{ll}%
0, & \text{if }x<-1/2\\
x+1/2, & \text{if }-1/2\leq x\leq1/2\\
1, & \text{if }x>1/2
\end{array}
\right.  \text{ \ \ }\left(  x\in\mathbb{R}\right)
\]
can be an example for a non-smooth sigmoidal function which satisfies $\left(
\Sigma3\right)  $ for all $\alpha>0.$ We further introduce a non-continuous
sigmoidal function $\sigma_{three}$ given by%
\[
\sigma_{three}:=\left\{
\begin{array}
[c]{ll}%
0, & \text{if }x<-1/2\\
1/2, & \text{if }-1/2\leq x\leq1/2\\
1, & \text{if }x>1/2
\end{array}
\right.  \text{ \ \ }\left(  x\in\mathbb{R}\right)  \text{ \ \cite{cost2}}%
\]
and a fifth variant $\sigma_{\gamma}$ ($0<\gamma\leq1$) defined by%
\[
\sigma_{\gamma}:=\left\{
\begin{array}
[c]{ll}%
\dfrac{1}{\left\vert x\right\vert ^{\gamma}+2}, & \text{if }x<-2^{1/\gamma}\\
2^{-\left(  1/\gamma\right)  -2}x+\left(  1/2\right)  , & \text{if
}-2^{1/\gamma}\leq x\leq2^{1/\gamma}\\
\dfrac{x^{\gamma}+1}{x^{\gamma}+2}, & \text{if }x>2^{1/\gamma}%
\end{array}
\right.  \text{ \ \ }\left(  x\in\mathbb{R}\right)  \text{ \ \cite{cost2}.}%
\]
Notice that $\sigma_{three}$ satisfies the property $\left(  \Sigma3\right)  $
for all $\alpha>0$, while $\phi_{\sigma_{\gamma}}\left(  x\right)  =O (
\left\vert x\right\vert ^{-1-\gamma} ) $ \ (see \cite{cost25}). Note also that
the kernels $\phi_{\sigma_{l}}$, $\phi_{\sigma_{h}}$ and $\phi_{\sigma
_{\gamma}}$ do not have compact support, whereas $\phi_{\sigma_{R}}$ and
$\phi_{\sigma_{three}}$ are compactly supported.

As a particular test function for our experiments, we consider the
discontinuous function $f:\left[  0,1\right]  \rightarrow\left[  0,1\right]  $
defined by%
\[
f\left(  x\right)  :=\left\{
\begin{array}
[c]{ll}%
0.2, & \text{if }0\leq x\leq0.2\\
0.9, & \text{if }0.2<x\leq0.5\\
0.3, & \text{if }0.5<x\leq0.8\\
0.6, & \text{if }0.8<x\leq1
\end{array}
\right.  \text{.}%
\]

\begin{figure}[ptb]
\centering
\includegraphics[height=5.1cm]{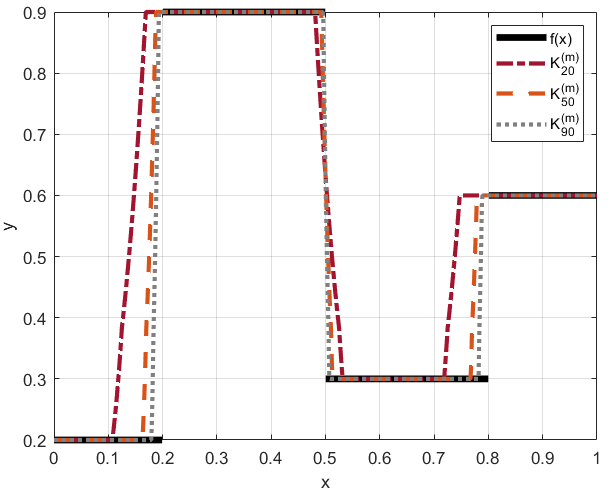}
\includegraphics[height=5.1cm]{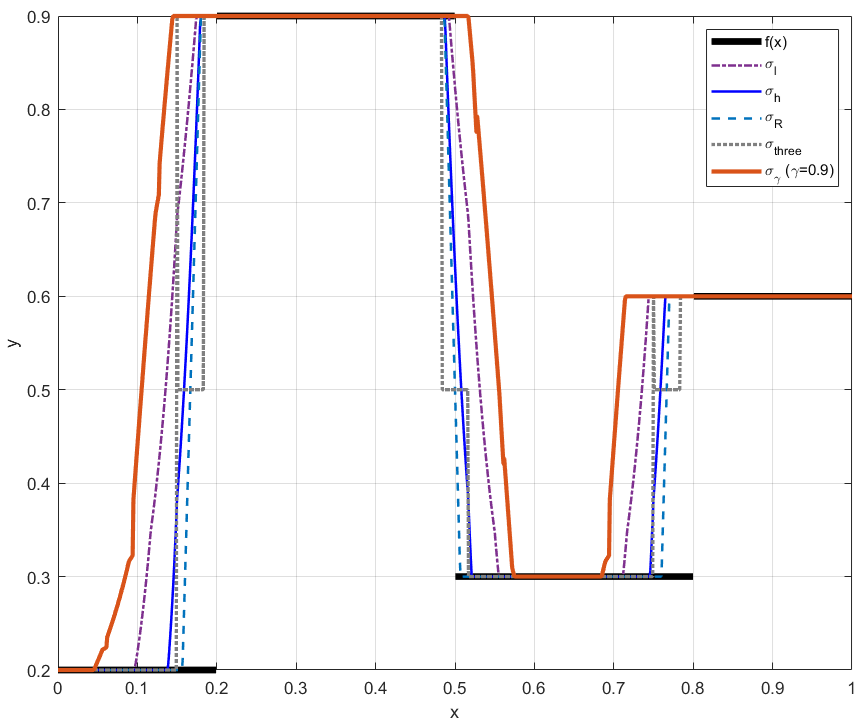} \caption{Left: Approximation
quality of Kantorovich type max-min operator $K_{n}^{(m)}(f)$ for increasing
sampling rate $n$. Right: Approximation of $f$ by $K_{30}^{(m)}(f)$ using
different sigmoidal functions.}%
\label{fig1}%
\end{figure}

Choosing $\sigma=\sigma_{h}$, the approximation by the Kantorovich type
max-min operator $K_{n}^{\left(  m\right)  }\left(  f\right)  $ is illustrated
in Figure \ref{fig1} (left) for increasing values of $n$. For the
discontinuous function $f$ it gets visible that point-wise convergence at the
discontinuity points can in general not be expected, but that $L^{p}%
$-convergence of $K_{n}^{\left(  m\right)  }\left(  f\right)  $ towards $f$ is
available. This is confirmed theoretically by Theorem \ref{theorem4} and gets
also apparent from the listed $L^{1}$-errors in Table \ref{table1}. In order
to give a broader picture about the approximation behavior of the different
kernels to the reader, we compare also the five introduced sigmoidal functions
($\sigma_{l},\sigma_{h},\sigma_{R},\sigma_{three},\sigma_{\gamma}$) for the
approximation of the same function $f$ taking $n=30$ sampling points. The
respective result is visualized in Figure \ref{fig1} (right).

\vspace{2mm}

\textbf{\noindent Comparison with max-product and linear neural network
operators.} In a next step, we want to compare the Kantorovich type max-min
operator with the Kantorovich variants of the max-product and the linear
neural network operators that are given as%
\[
K_{n}^{\left(  M\right)  }\left(  f;x\right)  :=\bigvee\limits_{k=\left\lceil
na\right\rceil }^{\left\lfloor nb\right\rfloor -1}n\int\limits_{\frac{k}{n}%
}^{\frac{k+1}{n}}f\left(  u\right)  du\dfrac{\phi_{\sigma}\left(  nx-k\right)
}{\bigvee\limits_{d=\left\lceil na\right\rceil }^{\left\lfloor nb\right\rfloor
-1}\phi_{\sigma}\left(  nx-d\right)  }\text{ \ \ (see \cite{cost25})}%
\]
and
\[
K_{n}\left(  f;x\right)  :=\sum\limits_{k=\left\lceil na\right\rceil
}^{\left\lfloor nb\right\rfloor -1}n\int\limits_{\frac{k}{n}}^{\frac{k+1}{n}%
}f\left(  u\right)  du\dfrac{\phi_{\sigma}\left(  nx-k\right)  }%
{\sum\limits_{d=\left\lceil na\right\rceil }^{\left\lfloor nb\right\rfloor
-1}\phi_{\sigma}\left(  nx-d\right)  }\text{ \ (see \cite{cost4})}%
\]
respectively. Taking these operators into account using the hyperbolic tangent
sigmoidal function $\sigma= \sigma_{h}$, we get the approximation errors in
the space $L^{1}\left(  \left[  0,1\right]  ,\left[  0,1\right]  \right)  $
listed in Table \ref{table1} and Figure \ref{fig2} (left). It can be seen,
that the Kantorovich variants of all three neural network operators behave
very similarly in terms of the approximation error with the max-product
variant having a slightly better performance than the max-min and the
classical linear variant. On the other hand, if we compare the computational
time of the operation, the max-min and max-product operators clearly
outperform the classical operator for sampling numbers $n>660$ (see Figure
\ref{fig2} (right)).

\begin{figure}[ptb]
\centering
\includegraphics[height = 4.2cm]{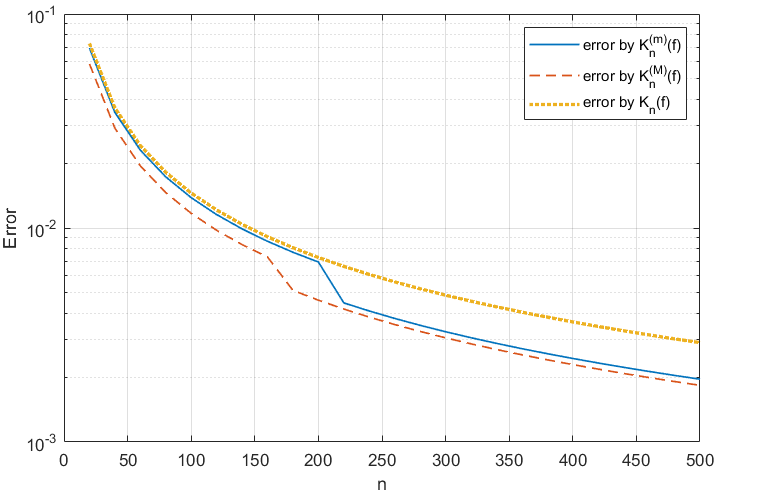}
\includegraphics[height = 4.2cm]{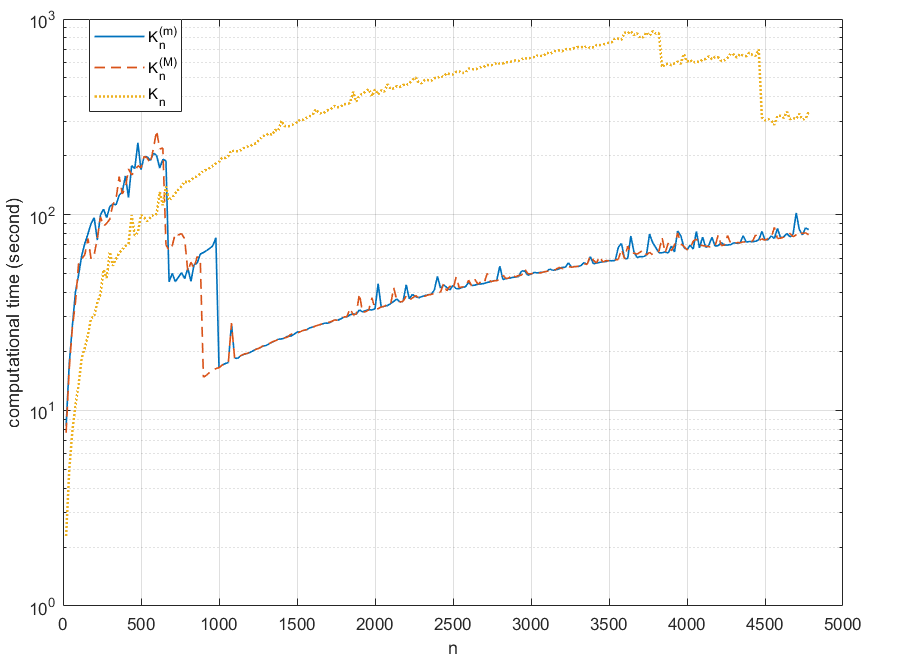} \caption{Left: errors for
the approximation of the function $f$ in $L^{1}\left(  \left[  0,1\right]
,\left[  0,1\right]  \right)  $ by $K_{n}^{(m)}(f),K_{n}^{(M)}(f)$ and
$K_{n}(f)$. Right: Comparison of the computational times to calculate the
functions $K_{n}^{(m)}(f)$, $K_{n}^{(M)}(f)$ and $K_{n}(f)$.}%
\label{fig2}%
\end{figure}%

\[
\overset{\text{\textbf{Table \ref{table1}: }Comparison of approximation errors
of\ Kantorovich type neural network operators}\ }{%
\begin{tabular}
[c]{cccc}\hline\hline
$n$ & $||K_{n}\left(  f\right)  -f||_{1}$ & $||K_{n}^{\left(  m\right)
}\left(  f\right)  -f||_{1}$ & $||K_{n}^{\left(  M\right)  }\left(  f\right)
-f||_{1}$\\\hline\hline
$10$ & 0.1457 & 0.1386 & 0.1171\\\hline
$30$ & 0.0485 & 0.0462 & 0.0390\\\hline
$90$ & 0.0162 & 0.0154 & 0.0130\\\hline
$150$ & 0.0097 & 0.0092 & 0.0078\\\hline
$500$ & 0.0029 & 0.0020 & 0.0018\\\hline\hline
\end{tabular}
\ \ \ }\label{table1}%
\]

\vspace{3mm}

\textbf{\noindent Application to signal denoising and ECG signal filtering.}
We finally provide some applications of the Kantorovich type max-min operators
to signal denoising. The incorporation of the Kantorovich information in the
max-min operator can be interpreted as a pre-processing step of the actual
max-min approximation in which a preliminary linear convolution filter is
applied to the initial signal. In general, this preliminary filtering allows
for an additional noise reduction and a smoothing of the signal.

To see the denoising benefits of the Kantorovich variant $K_{n}^{(m)}(f)$
compared to the classical sampling variant $F_{n}^{(m)}(f)$ of the neural
network operator, we apply both variants to the discontinuous test function
$f$ introduced above equipped with additional normally distributed Gaussian
noise. Using $n = 8000$ samples and a suitable scaled kernel $\phi_{\sigma
_{l}}(0.1 x)$ based on the logistic sigmoidal function $\sigma_{l}$, the
corresponding results are illustrated in Figure \ref{fig3}. In order to
approximate the local integrals of the noisy function and to calculate the
Kantorovich information, we used a Riemann sum based on additional refined
samples of the signal. With other quadrature rules as the trapezoidal rule,
similar results were obtained. It is visible that the usage of the Kantorovich
information of the signal instead of point evaluations leads to an improved
denoising of the noisy signal.

\begin{figure}[ptb]
\centering
\includegraphics[scale = 4]{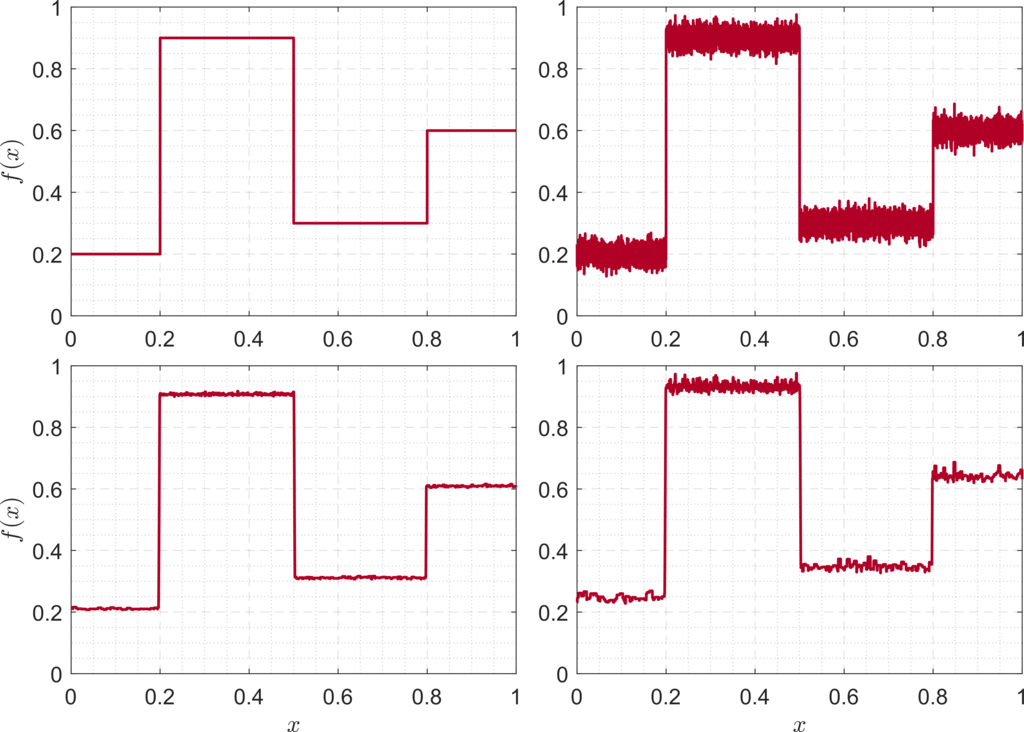}\caption{Top row: original
signal $f$ (left) and signal $f$ equipped with Gaussian noise (right). Bottom
row: Kantorovich variant $K_{n}^{(m)}(f)$ (left) compared to the classical max-min
variant $F_{n}^{(m)}(f)$ (right) of the max-min neural network operator
applied to the noisy signal.}%
\label{fig3}%
\end{figure}

For another comparison between the operators $K_{n}^{(m)}$ and
$K_{n}^{(M)}$, we apply both on an ECG signal describing $5$ heart beats
(using $n = 1600$ time samples) of patient 101 taken from the MIT-BIH
Arrhythmia Database \cite{moody}. Here we use the mean value of two consecutive time samples
to approximate the Kantorovich information and apply the operators
$K_{n/2}^{(m)}$ and $K_{n/2}^{(M)}$ to the ECG signal with the kernel $\phi_{\sigma_{l}}(2 x)$.
\begin{figure}[h]
\centering
\includegraphics[scale = 0.35]{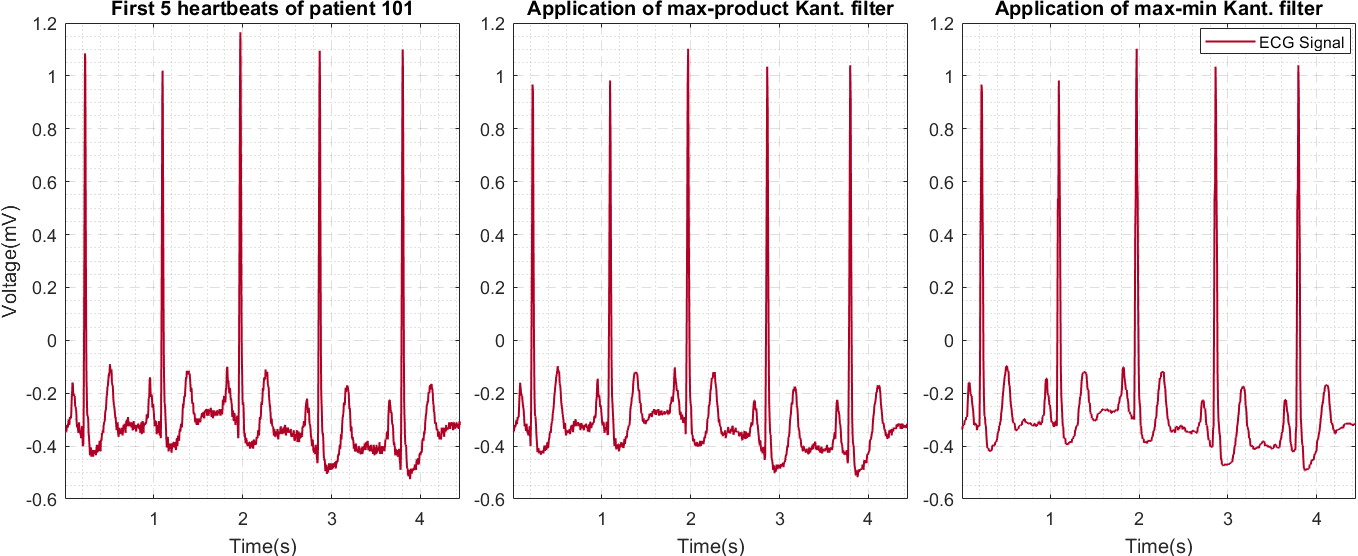}\caption{Comparison of the Kantorovich forms of the max-min and max-product Neural Network operators}%
\label{fig4}%
\end{figure}
Both operators provide denoised approximations of the ECG signal, the max-min Kantorovich variant having stronger smoothing effects, as shown in Figure \ref{fig4}.

\section{Concluding Remarks}

In this study, we investigated the Kantorovich variant of the max-min neural
network operators, analysing its convergence and approximation properties in
the $L^{p}$-spaces in more detail. In addition, we included some numerical
examples to underpin the theoretical results and to outline possible practical
applications of the operators for the denoising of biomedical signals. Our experiments demonstrate that max-min Kantorovich form of the network operators achieve superior results in denoising compared to their max-product counterparts. In the future, we aim to study the $N$-dimensional cases of these neural network operators, enabling us to analyze and process not only images but also higher-dimensional datasets.

\section{Acknowledgements}

The first author was funded by the Scientific and Technological Research
Council of T\"{u}rkiye (T\"{U}B\.{I}TAK). The second and third author were
funded by GNCS-IN$\delta$AM and the European Union - NextGenerationEU under
the National Recovery and Resilience Plan (NRRP), Mission 4 Component 2
Investment 1.1 - Call PRIN 2022 No. 104 of February 2, 2022 of Italian
Ministry of University and Research; Project 2022FHCNY3 (subject area: PE -
Physical Sciences and Engineering) "Computational mEthods for Medical Imaging
(CEMI)". We also thank the Italian Network on Approximation (RITA) and the
topical group on ``Approximation Theory and Applications'' of the Italian
Mathematical Union.

\medskip

\textbf{\.{I}smail Aslan}

Hacettepe University

Department of Mathematics,

\c{C}ankaya TR-06800, Ankara, T\"{u}rkiye

E-mail: ismail-aslan@hacettepe.edu.tr

\bigskip

\textbf{Stefano De Marchi}

University of Padova

Department of Mathematics ``Tullio Levi-Civita'',

Via Trieste 63, 35121 Padova, Italy

E-mail: stefano.demarchi@unipd.it

\bigskip

\textbf{Wolfgang Erb}

University of Padova

Department of Mathematics ``Tullio Levi-Civita'',

Via Trieste 63, 35121 Padova, Italy

E-mail: wolfgang.erb@unipd.it

\bigskip


\begin{thebibliography}{99}                                                                                               %


\bibitem {anas1}Anastassiou, G. A. (2011). Multivariate sigmoidal neural
network approximation. Neural Networks, 24(4), 378-386.

\bibitem {anas3}Anastassiou, G. A. (2011). Intelligent systems reference
library: vol. 19. Intelligent systems: approximation by artificial neural
networks, Berlin: Springer-Verlag.

\bibitem {anas4}Anastassiou, G. A. (2011). Multivariate hyperbolic tangent
neural network approximation. Computers \& Mathematics with Applications,
61(4), 809-821.

\bibitem {anas5}Anastassiou, G. A. (2011). Univariate hyperbolic tangent
neural network approximation. Mathematical and Computer Modelling, 53(5-6), 1111-1132.

\bibitem {anas2}Anastassiou, G. A. (2012). Univariate sigmoidal neural network
approximation. Journal of Computational Analysis \& Applications, 14(1), 659-690.

\bibitem {anas7}Anastassiou, G. A. (2019). Approximation by multivariate
sublinear and max-product operators. Revista de la Real Academia de Ciencias
Exactas, F\'{\i}sicas y Naturales. Serie A. Matem\'{a}ticas, 113, 507-540.

\bibitem {is}Aslan, I. (2024). Approximation by Max-Min Neural Network
Operators. (submitted for publication).

\bibitem {bede3}Bede, B., Coroianu, L., \& Gal, S. G. (2009). Approximation
and shape preserving properties of the Bernstein operator of max-product kind.
International journal of mathematics and mathematical sciences, 2009.

\bibitem {bede4}Bede, B., Coroianu, L., \& Gal, S. G. (2010). Approximation
and shape preserving properties of the nonlinear Favard-Szasz-Mirakjan
operator of max-product kind. Filomat, 24(3), 55-72.

\bibitem {bede5}Bede, B., Coroianu, L., \& Gal, S. G. (2016). Approximation by
max-product type operators. Heidelberg: Springer.

\bibitem {bede1}Bede, B., Nobuhara, H., Da\v{n}kov\'{a}, M., \& Di Nola, A.
(2008). Approximation by pseudo-linear operators. Fuzzy Sets and Systems,
159(7), 804-820.

\bibitem {bede2}Bede, B., Schwab, E. D., Nobuhara, H., \& Rudas, I. J. (2009).
Approximation by Shepard type pseudo-linear operators and applications to
image processing. International journal of approximate reasoning, 50(1), 21-36.

\bibitem {butzer}Butzer, P. L., \& Berens, H. (2013). Semi-groups of operators
and approximation (Vol. 145). Springer Science \& Business Media.

\bibitem {costez}Costarelli, D, (2014). Sigmoidal functions approximation and
applications (Ph.D. thesis), Roma Tre University, Rome, Italy.

\bibitem {est3}Costarelli, D., \& Sambucini, A. R. (2018). Approximation
results in Orlicz spaces for sequences of Kantorovich max-product neural
network operators. Results in Mathematics, 73(1), 15.

\bibitem {cost8}Costarelli, D., Sambucini, A. R., \& Vinti, G. (2019).
Convergence in Orlicz spaces by means of the multivariate max-product neural
network operators of the Kantorovich type and applications. Neural Computing
and Applications, 31(9), 5069-5078.

\bibitem {cost1}Costarelli, D., \& Spigler, R. (2013). Approximation results
for neural network operators activated by sigmoidal functions. Neural
Networks, 44, 101-106.

\bibitem {cost3}Costarelli, D., \& Spigler, R. (2013). Multivariate neural
network operators with sigmoidal activation functions. Neural Networks, 48, 72-77.

\bibitem {cost4}Costarelli, D., \& Spigler, R. (2014). Convergence of a family
of neural network operators of the Kantorovich type. Journal of Approximation
Theory, 185, 80-90.

\bibitem {cost2}Costarelli, D., \& Vinti, G. (2016). Max-product neural
network and quasi-interpolation operators activated by sigmoidal functions.
Journal of Approximation Theory, 209, 1-22.

\bibitem {cost25}Costarelli, D., \& Vinti, G. (2016). Approximation by
max-product neural network operators of Kantorovich type. Results in
Mathematics, 69, 505-519.

\bibitem {cost5}Costarelli, D., \& Vinti, G. (2016). Pointwise and uniform
approximation by multivariate neural network operators of the max-product
type. Neural Networks, 81, 81-90.

\bibitem {cost7}Costarelli, D., \& Vinti, G. (2017). Saturation classes for
max-product neural network operators activated by sigmoidal functions. Results
in Mathematics, 72, 1555-1569.

\bibitem {est1}Costarelli, D., \& Vinti, G. (2018). Estimates for the neural
network operators of the max-product type with continuous and p-integrable
functions. Results in Mathematics, 73, 1-10.

\bibitem {est2}Costarelli, D., \& Vinti, G. (2019). Quantitative estimates
involving K-functionals for neural network-type operators. Applicable
Analysis, 98(15), 2639-2647.

\bibitem {cost9}Cantarini, M., Coroianu, L., Costarelli, D., Gal, S. G., \&
Vinti, G. (2021). Inverse result of approximation for the max-product neural
network operators of the Kantorovich type and their saturation order.
Mathematics, 10(1), 63.

\bibitem {carda}Cardaliaguet, P., \& Euvrard, G. (1992). Approximation of a
function and its derivative with a neural network. Neural networks, 5(2), 207-220.

\bibitem {cheang}Cheang, G. H. (2010). Approximation with neural networks
activated by ramp sigmoids. Journal of Approximation Theory, 162(8), 1450-1465.

\bibitem {cao}Chen, Z., \& Cao, F. (2009). The approximation operators with
sigmoidal functions. Computers \& Mathematics with Applications, 58(4), 758-765.

\bibitem {cao2}Chen, Z., \& Cao, F. (2012). The construction and approximation
of a class of neural networks operators with ramp functions. Journal of
Computational Analysis and Applications, 14(1), 101-112.

\bibitem {coroianu4}Coroianu, L., Costarelli, D., \& Kadak, U. (2022).
Quantitative estimates for neural network operators implied by the asymptotic
behaviour of the sigmoidal activation functions. Mediterranean Journal of
Mathematics, 19(5), 211.

\bibitem {coroianu1}Coroianu, L., Costarelli, D., Gal, S. G., \& Vinti, G.
(2019). The max-product generalized sampling operators: convergence and
quantitative estimates. Applied Mathematics and Computation, 355, 173-183.

\bibitem {coroianu2}Coroianu, L., \& Gal, S. G. (2018). Approximation by
truncated max-product operators of Kantorovich-type based on generalized
($\Phi,\Psi$)-kernels. Mathematical Methods in the Applied Sciences, 41(17), 7971-7984.

\bibitem {coroianu3}Coroianu, L., \& Gal, S. G. (2022). New approximation
properties of the Bernstein max-min operators and Bernstein max-product
operators. Mathematical Foundations of Computing, 5(3), 259-268.

\bibitem {duman2}Duman, O. (2016). Nonlinear Approximation: q-Bernstein
Operators of Max-Product Kind. In Intelligent Mathematics II: Applied
Mathematics and Approximation Theory (pp. 33-56). Springer International Publishing.

\bibitem {duman}Duman, O. (2024). Max-product Shepard operators based on
multivariable Taylor polynomials. Journal of Computational and Applied
Mathematics, 437, 115456.

\bibitem {aslan}G\"{o}k\c{c}er, T. Y., \& Aslan, \.{I}. (2022). Approximation
by Kantorovich-type max-min operators and its applications. Applied
Mathematics and Computation, 423, 127011.

\bibitem {gokcer2}G\"{o}k\c{c}er, T. Y., \& Duman, O. (2020). Approximation by
max-min operators: A general theory and its applications. Fuzzy Sets and
Systems, 394, 146-161.

\bibitem {gokcer}G\"{o}k\c{c}er, T. Y., \& Duman, O. (2022). Regular
summability methods in the approximation by max-min operators. Fuzzy Sets and
Systems, 426, 106-120.

\bibitem {holhos1}Holho\c{s}, A. (2018). Weighted approximation of functions
by Meyer--K\"{o}nig and Zeller operators of max-product type. Numerical
Functional Analysis and Optimization, 39(6), 689-703.

\bibitem {kad}Kadak, U. (2022). Multivariate neural network interpolation
operators. Journal of Computational and Applied Mathematics, 414, 114426.

\bibitem {moody}Moody, G. B., Mark, R. G. (2001). The impact of the MIT-BIH
Arrhythmia Database. IEEE Eng. in Med. and Biol., 20(3), 45-50

\bibitem {peetre}Peetre, J. (1970). A new approach in interpolation spaces.
Studia Mathematica, 34(1), 23-42.

\bibitem {yuksel}Y\"{u}ksel G\"{u}ng\"{o}r, \c{S}., \& \.{I}spir, N. (2018).
Approximation by Bernstein-Chlodowsky operators of max-product kind.
Mathematical Communications, 23(2), 205-225.
\end{thebibliography}
\end{document}